\documentclass[journal=iecred,manuscript=article]{achemso}
\setkeys{acs}{doi = true}
\usepackage{chemformula} 
\usepackage[T1]{fontenc} 

\usepackage{lineno,hyperref}
\modulolinenumbers[5]

\usepackage{amsmath,amssymb,amsfonts}
\usepackage{amsthm}
\usepackage[ruled,linesnumbered]{algorithm2e}
\usepackage{subcaption}

\usepackage{geometry}
\usepackage{color}
\usepackage{easyReview}
\SectionNumbersOn

\newtheorem{remark}{Remark}

\newtheorem{theorem}{Theorem}
\newtheorem{corollary}{Corollary}
\newtheorem{lemma}[theorem]{Lemma}
\newtheorem{definition}{Definition}
\newtheorem{assumption}{Assumption}

\usepackage{multirow}

\newcommand{\Real}{\mathbb R}
\def\d{\mathrm{d}}

\title{Generalized Global Self-Optimizing Control for Chemical Processes: Part II. Objective-guided Controlled Variable Learning Approach}

\author{Chenchen Zhou}
\affiliation[Zhejiang University]
{College of Chemical and Biological Engineering, Zhejiang University, Hangzhou, Zhejiang 310058, China}
\alsoaffiliation[Quzhou Institute]
{Institute of Zhejiang University-Quzhou, Quzhou, Zhejiang 324000, China}

\author{Hongxin Su}
\affiliation[Zhejiang University]
{College of Chemical and Biological Engineering, Zhejiang University, Hangzhou, Zhejiang 310058, China}

\author{Xinhui Tang}
\affiliation[Zhejiang University]
{College of Chemical and Biological Engineering, Zhejiang University, Hangzhou, Zhejiang 310058, China}

\author{Yi Cao}
\affiliation[Zhejiang University]
{College of Chemical and Biological Engineering, Zhejiang University, Hangzhou, Zhejiang 310058, China}

\author{Shuang-Hua Yang}
\affiliation[University of Reading]
{Department of Computer Science, University of Reading, Whiteknights, RG6 6UR, United Kingdom}
\alsoaffiliation[Quzhou Institute]
{Institute of Zhejiang University-Quzhou, Quzhou, Zhejiang 324000, China}
\email{shuang-hua.yang@reading.ac.uk}

\author{Lingjian Ye}
\affiliation[Huzhou University]
{School of Engineering, Huzhou University, Huzhou 313000, China}
\email{lingjian.ye@zjhu.edu.cn}


\begin{document}

		\begin{abstract}
			{Self-optimizing control (SOC) aims to maintain near-optimal process operation by judiciously selecting controlled variables (CVs).
			In this series of work, the generalized global SOC ($g^2$SOC) approach is proposed, which extends the concept of SOC to the whole operation space and uses general nonlinear functions to design CVs instead of linear combinations. In the first part of this series work, two numerical approaches for $g^2$SOC are proposed: the optimization-based approach and the regression-based approach, based on a theoretical analysis of the existence of perfect self-optimizing CVs. The CVs designed by the former perform better, but are usually infeasible for large-scale problems.
			In this paper, we propose an algorithm called objective-guided controlled variable learning (OGCVL) that combines the advantages of both and has a better scalability. OGCVL is proposed for efficient CV design that seamlessly integrates symbolic and numerical computation techniques. Finally, the effectiveness of the OGCVL method is verified in two numerical examples. Both examples illustrate show that the OGCVL method is able to achieve good results while maintaining computational efficiency and is also feasible in large-scale problems.}

		\end{abstract}

	\section{Introduction}
	{The optimization of chemical processes under uncertainty remains a critical challenge in the industry. Real-time optimization (RTO) has emerged as a widely adopted technology for maintaining processes close to their optimal operating points under changing conditions \cite{chachuatAdaptationStrategiesRealtime2009}. However, traditional RTO approaches, such as the two-step RTO method \cite{marlin1997real}, often suffer from long waiting times between steady states, resulting in suboptimal operation and economic losses \cite{srinivasan2019a}.}
	
	{Self-Optimizing Control (SOC) has been proposed as an alternative approach to address these limitations \cite{halvorsen:selfopt2003}. SOC extends the concept of controlled variables (CVs) from single measurement variables to functions of measurement variables\cite{halvorsen:selfopt2003}, providing a systematic framework for designing CVs that maintain near-optimal operation without the need for frequent re-optimization.}
	
	{Local SOC methods, based on local approximations of the economic objective and plant model, have been extensively studied \cite{Alstad:optcomb2007,Kariwala:OptComb,KariwalaCao2008}. However, these methods assume minimal disturbance changes and may lead to significant economic losses when the plant operates far from the nominal point. To address this issue, global SOC (gSOC) methods have been developed for wide operating windows \cite{YeCao2015gSOC,YeCao2022gSOC_CCE}, enhancing SOC performance across the entire disturbance space at the cost of increased computational complexity. In addition, local analysis actually limits the exploration of more advanced forms of CVs.  Nonlinear combinations of measurements as control variables have been shown to achieve better performance in nonlinear systems\cite{Jaschke2012Polynomial,YeCao2013IECR,Su2022}. However, these methods are either limited to analyzing specific nonlinear systems or focus solely on algorithm design, without conducting a global analysis of the optimization problem for general nonlinear systems, thus failing to truly elucidate and harness the potential of gSOC.}
	
	{Motivated by those limitations of local analysis, this work performs a global analysis of the steady-state process optimization problem for generalized chemical processes, and consequently proposes the generalized global Self-Optimizing Control ($g^2$SOC) approach. Part I of this work \cite{ye2023} proves the existence of perfect global self-optimizing controlled variables under proper conditions, and two numerical approaches are proposed to approximate the perfect self-optimizing controlled variables, which are regression-based and optimization-based approaches.}
	
	{The regression-based approach can be considered as optimizing first, then training. It involves solving the RTO problem under various uncertainty scenarios to generate training data, i.e., the optimal measurements and control inputs. Subsequently, through a comparative learning approach, deep neural networks are employed to approximate the perfect self-optimizing CVs. This approach is easy to implement, however, since the training process solely focuses on the data itself, neglecting the underlying physical laws that should be satisfied, the closed-loop performance is typically unsatisfactory.}
	
	{The optimization-based approach can be regarded as simultaneously optimizing and training by directly solving a massive-scale nonconvex nonlinear programming problem. Since the model (physical laws) is explicitly considered during the optimization training process, if the solution is successful, the closed-loop performance of the obtained CVs is generally better. However, the main drawback of this approach is that if the considered system is complex or large-scale, it is typically impossible to solve within a limited time. Additionally, due to the nonconvex nature of the problem, successful solution is challenging. Similar issues also arise in the domain of model predictive control policy approximation \cite{li2022} and the solution of partial differential equations \cite{Raissi2019}.}
	
	The objective of this work is to propose a framework that lies between the optimization and regression approaches, capable of combining the advantages of both approaches. This framework aims to design CVs in a computationally efficient manner while ensuring near optimal closed-loop performance.
	In this paper, thorough analysis and comparison of the optimization-based and regression-based approaches are first presented, highlighting the advantages and characteristics of each method. Subsequently, the objective-guided controlled variable learning approach for the $g^2$SOC method is proposed, fully incorporating the strengths of both optimization-based and regression-based methods. The contributions of this paper are summarized as follows:
	\begin{enumerate}
		\item Under more lenient assumptions in Part I of this endeavor, a comprehensive global analysis of the steady-state optimization problem is undertaken, broadening the prerequisites for the existence of perfect global self-optimizing CVs and furnishing a more generalized formulation of perfect global self-optimizing CVs. 
		\item Theoretical comparisons between optimization-based and regression-based approaches are developed.
		\item An objective-guided controlled variable learning framework is proposed to design CVs, which connects symbolic computation and numerical computation, and embeds domain knowledge (process model) as well as the information of the control objective into the CVs model training process.	
		\item The effectiveness of the proposed algorithm is demonstrated through a small and a large-scale example.
	\end{enumerate}

	
	
	
	The remainder of this paper is organized as follows. Section~\ref{sec:RTO} briefly reviews optimization problem of uncertain processes and self-optimizing control. The existence theorem of perfect global SOC CVs is also proposed in this section.
	The two numerical controlled variable design methods, regression and optimization methods, proposed in Part I of this work \cite{ye2023} are deeply analyzed to compare the pros and cons of both in Section~\ref{sec:relationship}.
	In Section \ref{sec:g2SOC}, an objective-guided controlled variable learning framework is proposed to design CVs.
	In Section~\ref{sec:case}, two examples are studied to illustrate the effectiveness of the proposed method, and the final section presents conclusions and outlooks.

	\section{Optimization under uncertainty and perfect global self-optimizing controlled variables}
	\label{sec:RTO}
	\subsection{Optimization Problem of Steady-State Operation under uncertainty}
	Consider a static operation optimization problem with uncertainty:
	\begin{equation}  
		\label{eq:sopt}
		\begin{aligned}
			&\min _{\mathbf{u}} J(\mathbf{u}, \mathbf{d}) \\
			&\text { s.t. }  g_i(\mathbf{u}, \mathbf{d}) \leq 0  ,\;\;\; i=1,\ldots,n_g \\
			& \qquad \mathbf{y} = \mathbf{m}(\mathbf{u}, \mathbf{d}) \\
			& \qquad \mathbf{u} \in \mathcal{U} \quad \mathbf{d} \in \mathcal{D} \\
		\end{aligned}
	\end{equation}
	where $J: \Real^{n_{u}} \times \Real^{n_{d}} \mapsto \Real$ is the steady state economic cost function, $ \mathbf{d} \in \Real^{n_d} $ stands for disturbances, and $ \mathbf{u} \in \Real^{n_u}$ stands for the steady state degrees of freedom, $g_i: \Real^{n_{u}} \times \Real^{n_{d}} \mapsto \Real, i=1,\ldots,n_g$ are the operational constraints,
	$\mathbf{m}: \Real^{n_{u}} \times \Real^{n_{d}} \mapsto \Real^{n_{y}}$ describes the relationship between $ \mathbf{u}, \mathbf{d} $ and system output $ \mathbf{y} \in \Real^{n_{y}} $. 
	$ \mathcal{U} $ and $ \mathcal{D} $ denote the feasible region of $ \mathbf{u} $ and the domain of $ \mathbf{d} $, which are assumed compact. The model mismatch can also be regarded as a disturbance. Furthermore, $ \mathbf{d} $ are assumed as independent and identically distributed (i.i.d.) variables. 
	
	Before further, some necessary definition and assumptions on Problem \eqref{eq:sopt} will be introduced.
	\begin{definition}[Continuity and Semicontinuity of Sets]\cite{dontchev2014a}
		A set $S: \mathbb{R}^m \rightrightarrows \mathbb{R}^n$ is deemed outer semicontinuous (osc) at $\bar{y}$ if it satisfies:
		$$
		\limsup_{y \rightarrow \bar{y}} S(y) \subset S(\bar{y})
		$$
		
		A set $S: \mathbb{R}^m \rightrightarrows \mathbb{R}^n$ is deemed inner semicontinuous (isc) at $\bar{y}$ if it satisfies:
		$$
		\liminf_{y \rightarrow \bar{y}} S(y) \supset S(\bar{y})
		$$
		
		A set $S: \mathbb{R}^m \rightrightarrows \mathbb{R}^n$ is Pompeiu–Hausdorff continuous at $\bar{y}$ if it is simultaneously outer semicontinuous and inner semicontinuous at $\bar{y}$, and $S(\bar{y})$ is a closed set.
	\end{definition}
	\begin{assumption}
		\label{ass:1}
		$\forall$ $g_i$ is a continuous real-valued function on $\mathcal{U} \times \mathcal{D}$, and for any $\mathbf{d}$, each $g_i$ is convex in $\mathbf{u}$.
	\end{assumption}
	
	\begin{assumption}
		\label{ass:2}
		$\forall$ $\mathbf{d} \in \mathcal{D}$, $\exists \mathbf{u} \in \mathcal{U} $ such that $ g_i(\mathbf{u}, \mathbf{d}) \leq 0 ,i=1,\ldots,n_g.$
	\end{assumption}
	
	Assumption \ref{ass:1} imposes the convexity requirement upon the constraint functions, while Assumption \ref{ass:2} stipulates a non-empty feasible domain comprising more than a single point. These assumptions are relatively commonplace in the realm of optimization problems.
	
	\begin{lemma}\cite{dontchev2014a}
		Upon satisfying Assumptions \ref{ass:1} and \ref{ass:2}, the feasible set of Problem \eqref{eq:sopt} can be represented as:
		$$
		S_{\mathrm{feas}}: \mathbf{d} \mapsto \left\{\mathbf{u} | g_i(\mathbf{u}, \mathbf{d}) \leq 0 ,i=1,\ldots,n_g \right\} \quad \mathbf{d} \in \mathcal{D}
		$$
		which is Pompeiu–Hausdorff continuous on $\mathbf{d} \in \mathcal{D}$. This implies the continuity of the feasible domain's boundary.
	\end{lemma}
	
	\begin{proof}
		See \cite{dontchev2014a}, Example 3B.4.
	\end{proof}
	
	\begin{assumption}
		\label{ass:33}
		$J$ is a continuous real-valued function on $\mathcal{U} \times \mathcal{D}$.
	\end{assumption}
	
	This assumption is rather rudimentary, and does not necessitate the convexity of $J$.
	
	To facilitate exposition, we introduce the mappings for the optimal value and optimal solutions.
	
	The optimal value mapping, acting from $\mathcal{D}$ to $\mathbb{R}$, is defined by:
	$$S_{\mathrm{val}} : \mathbf{d} \mapsto \inf_s \left\{J(\mathbf{u}, \mathbf{d}) \mid \mathbf{u} \in S_{\text {feas}}(\mathbf{d})\right\} \text{ when the inf is finite}$$
	
	The optimal set mapping, acting from $\mathcal{D}$ to $\mathcal{U}$, is defined by:
	$$
	S_{\mathrm{opt}}: \mathbf{d} \mapsto\left\{\mathbf{d} \in S_{\text {feas }}(\mathbf{d}) \mid J(\mathbf{u}, \mathbf{d})=S_{\mathrm{val}}(\mathbf{d})\right\}
	$$
	
	\begin{theorem}[Fundamental Continuity Properties of the Mappings for Optimization Solutions]	\cite{dontchev2014a}
		\label{the:opt}		
		If for all $\mathbf{d} \in \mathcal{D}$, $S_{\mathrm{feas}}$ is non-empty and bounded, and Assumptions \ref{ass:1}, \ref{ass:2}, and \ref{ass:33} are satisfied, then
		
		$S_{\mathrm{val }}$ is continuous on $\mathbf{d} \in \mathcal{D}$, and $S_{\mathrm{opt}}$ is outer semicontinuous on $\mathbf{d} \in \mathcal{D}$.	
	\end{theorem}
	\begin{proof}
		See \cite{dontchev2014a}, Theorem 3B.5.
	\end{proof}
	
	\begin{assumption}
		\label{ass:44}
		When employing a specific numerical method to solve the problem \eqref{eq:sopt}, with appropriate solver settings and solution strategies, an optimal solution can be obtained for all $\mathbf{d} \in \mathcal{D}$.
	\end{assumption}
	
	If Assumption \ref{ass:44} holds, the following mapping  $\pi:\Real^{n_d} \mapsto \Real^{n_u}$ exists:
	\begin{equation}
		\label{eq:uopt_of_d}
		\mathbf{u}^{\mathrm{opt}} = \pi(\mathbf{d}) \in S_{\mathrm{opt}}(d)
	\end{equation}
	where $\mathbf{u}^{\mathrm{opt}}$ denotes the optimal control input, procured through the implementation of a numerical solver.
	Based on Theorem \ref{the:opt}, $S_{\mathrm{opt}}$ is outer semicontinuous on $\mathbf{d} \in \mathcal{D}$, implying that the points contained in $S_{\mathrm{opt}}(\mathbf{d})$ may not be unique. For practical problems, this suggests the potential existence of multiple optimal solutions. Furthermore, $ \pi(\mathbf{d})$ can be viewed as a sampling from $S_{\mathrm{opt}}(\mathbf{d})$, indicating that $\pi$ may be a discontinuous function on ${\mathbf{d}} \in \mathcal{D}$, with no guarantee of the set of optimal solutions being a continuous function of $\mathbf{d}$. 
	
	In summary, for optimization problems satisfying the aforementioned assumptions, the solutions obtained through numerical methods may exhibit discontinuous behavior, with the locations of discontinuities being unknown.

	Thus far, we have elucidated that for general steady-state optimization problems, an inherent mapping relationship exists between the optimal control inputs $\mathbf{u}^{\mathrm{opt}}$ and the disturbances $\mathbf{d}$. However, the optimal solutions cannot be directly obtained due to the unknown and dynamically evolving nature of the disturbances $\mathbf{d}$.
	
	In the traditional two-step RTO approaches, the measurements are utilized to refine the process model, i.e. estimate $\mathbf{d}$, and the updated model is subsequently employed for optimization purposes.
	This approach is similar to feedforward control, and correspondingly, the SOC approaches do not explicitly estimate the uncertainties, which constitutes the distinguishing feature between them.
	The SOC approach furnishes a means of transmuting the optimization problem into a feedback control problem.
	The challenge and core of SOC lie in finding CVs $\mathbf{c}=h(\mathbf{y})$ (functions of the measured variables) which, when kept constant by adjusting the input variables, restore the optimality of plant performance. For instance, the active constraints, $\mathbf{g}_a$, with strict equality maintained throughout the disturbance domain, could constitute a form of SOC CVs.
	
	In fact, both approaches implicitly assume that $\mathbf{d}$ can be estimated from the measured variables $\mathbf{y}$ or that the measurements contain the information pertaining to disturbances.
	In other words, they infer that the information contained in measured variables is sufficient for achieving optimal operation in steady state.
	
	\subsection{Necessary conditions for system to achieve optimal operation through output feedback}
	\label{sec:FCAOSSO}
	In the ensuing discussion, we shall investigate the requisite conditions for attaining optimal steady-state operation predicated upon the system outputs, that is to say, the information contained within the measured variables. To initiate, we proffer the definition of optimal steady-state operation achievable based on system outputs:
	\begin{definition}
		If a mapping $\varpi$ exists such that $\mathbf{u}^{\mathrm{opt}} = \varpi(\mathbf{y})$, then it is said that optimal steady-state operation can be achieved based on system outputs.
	\end{definition}
	In the following, the inverse function theorem is invoked for theoretical analysis of feasibility conditions for achieving optimal steady-state operation based on measured variables.
	\begin{theorem}(Hadamard's global inverse function theorem, Theorem 6.2.3 in \cite{krantz2002implicit}). \label{theo:gInvFunc}
		\\
		Let $f:\;\mathcal{R}^{n_x}\rightarrow \mathcal{R}^{n_x}$ be a $\rm{C}^2$ mapping. Suppose that $f(0)=0$ (without loss of generality) and the Jacobian determinant of $f$, $\mid\frac{\partial f}{\partial x}\mid\neq 0$, is nonzero at each point. Furthermore, suppose that $f$ is proper. Then $f$ is one-to-one and onto.
	\end{theorem}
	For chemical processes, it is common to assume the process model $ \mathbf{m} $ is a $ C^2 $ mapping and the Jacobian determinant of $\mathbf{m}$, $\mid\frac{\partial \mathbf{m}}{\partial [\mathbf{u}, \mathbf{d}]}\mid$, is nonzero, which means all inputs and disturbances have an effect on the output/measurements.
	 	
	Theorem \ref{theo:gInvFunc} shows that there exists a continuously differentiable inverse function for $ \mathbf{y} = \mathbf{m}(\mathbf{u}, \mathbf{d}) $ when $n_y=n_u+n_d$ (equivalent inputs and outputs), denoted by $ [\mathbf{u}, \mathbf{d}] = \mathcal{F}(\mathbf{y}) $. And when $n_y \geq n_u+n_d$, such inverse mappings could be non-unique. 
	\begin{corollary}
		\label{theo:n_x < n_y}
		Let the mapping $ f: \Real^{n_x} \mapsto \Real^{n_y}, n_x < n_y $ be a $\rm{C}^2$ mapping. Suppose that $f(0)=0$ (without loss of generality) and the Jacobian of $f$, $\frac{\partial f}{\partial x}$, is full rank and the rank is equal to $ n_x $ at each point. Furthermore, suppose that $f$ is proper. Then the following mapping exists:
			$$x = \mathcal{F}(y)$$	
	\end{corollary}
	
	\begin{proof}
		For the mapping $ f: \Real^{n_x} \mapsto \Real^{n_y} $ when $ n_x < n_y  $, $ A \triangleq \frac{\partial f}{\partial x} \in \Real^{n_y \times n_x} $ is full rank with $ rank(A)=n_x $. So that 
			$$A^{\top} y = A^{\top}f(x).$$
		And the Jacobian of $A^{\top} f(x)$ is $ A^{\top}A $ which is full rank with $rank( A^{\top}A) = n_x $,
		satisfying assumptions in Theorem \ref{theo:gInvFunc}, then the following equation holds

			$$x = \mathcal{F}_A(A^{\top}y)$$
		where $ \mathcal{F}_A $ is the inverse function to $ A^{\top}y $. So the mapping $ \mathcal{F}_A(A^{\top}.) $ is a selection of $ \mathcal{F} $.
	\end{proof}
%
	
	In any case, where  $ n_y \geq n_d +n_u$, and under conditions of Theorem~\ref{theo:gInvFunc}, there exit inverse mappings that can precisely compute $ \mathbf{u} $ and $ \mathbf{d} $ using the information of $ \mathbf{y} $, denoted by
	\begin{equation}
		\label{eq:d=f(y)}
		\mathbf{d} = \mathcal{F}_d(\mathbf{y})
	\end{equation}
	\begin{equation}
		\label{eq:u=f(y)}
		\mathbf{u} = \mathcal{F}_u(\mathbf{y})
	\end{equation}
	By substituting Eq.~\eqref{eq:d=f(y)} into Eq.~\eqref{eq:uopt_of_d}, it yields
	\begin{equation}
		\label{eq:uopt_of_y}
		\mathbf{u}^{\mathrm{opt}} = \pi(\mathbf{d}) = \pi(\mathcal{F}_d(\mathbf{y}))
	\end{equation}
	so $\mathbf{u}^{\mathrm{opt}} =\varpi(\mathbf{y})$ exists.
	
	 
	 \begin{theorem}(Necessary conditions for system to achieve optimal operation through output feedback)\\ \label{the:pCVs}
	 	Let the conditions required by Theorem~\ref{theo:gInvFunc}, Assumption \ref{ass:1}, \ref{ass:2}, \ref{ass:33} and \ref{ass:44} be satisfied, then the optimal steady-state operation could be achieved based on the information contained in measurements.
	 \end{theorem}
 	\begin{proof}
 		If Assumption \ref{ass:1}, \ref{ass:2}, \ref{ass:33} and \ref{ass:44} are satisfied, the mapping $\pi$ in Eq.~\eqref{eq:uopt_of_d} exists.
 		If the conditions required by Theorem~\ref{theo:gInvFunc} hold for process model $ \mathbf{m} $, then the mapping $\mathcal{F}_d$ in Eq.~\eqref{eq:d=f(y)} and the mapping $\mathcal{F}_u$ in Eq.~\eqref{eq:u=f(y)} exist.
 		Then $\mathbf{u}^{\mathrm{opt}} $ can be represented by the function of $ \mathbf{y} $, like Eq.~\eqref{eq:uopt_of_y}. 
 		
 		Obviously, keeping $ \mathbf{u} =\mathbf{u}^{\mathrm{opt}} $ by adjusting $ \mathbf{u} $ will achieve the optimal steady-static operation. And $\mathbf{u}^{\mathrm{opt}} $ could be represented by the functions of $ \mathbf{y} $.
 	\end{proof}
	\begin{remark}
		This condition cannot be satisfied if the measurement noise is also considered as a disturbance, since the number of measurements $n_y$ must be smaller than the number of disturbances $n_d$. Because a steady state process is considered, one way to do this is to consider noise reduction of the measured data using some filters, as further described in the subsequent case study.
		Another perspective is that measurement noise does not affect the optimal control inputs and therefore there is no need to include this information in the controlled variables. This work focuses on the ideal, noise-free case, the extension to noisy measurements would require these additional considerations.
	\end{remark}
	\subsection{Perfect global self-optimizing controlled variables}
	Theorem~\ref{the:pCVs} has delineated the necessary conditions for system outputs to achieve optimal operation through feedback.
	Optimal operation is rendered possible only if the output of the system contains information about all disturbances that impact the objective.
	Therefore, when the conditions of this theorem are satisfied, both self-optimizing control and RTO can theoretically attain optimal operation.
	However, the key distinction between SOC and RTO lies in the former being a feedback control strategy and the latter being a feedforward control strategy.
	The SOC strategy aims to achieve optimization through feedback control by selecting appropriate CVs. It does not explicitly estimate uncertainty.
	The feedback controller design of SOC is segregated from the CVs design. In essence, the CVs represent an invariant feature of the steady-state optimal operation of a system. SOC transfers an optimization problem into a tracking control problem, which renders the online computation cost for SOC lower than RTO. However, this creates additional requirements for SOC—Closed-loop achievability.
	
	However, in traditional SOC methods, this factor is usually not taken into account. They only consider the economic performance of the close loop system. Commonly, performance evaluation metrics can be quantified as either the worst-case loss or the average loss, as expressed by the following equations:
	\begin{equation}
		L_{c, \text{w c}}=\underset{\mathbf{d} \in \mathcal{D}}\max  L_c
	\end{equation}
	\begin{equation}
		L_{c, \text{a v}}=\underset{\mathbf{d} \in \mathcal{D}}{\mathbb{E}}\left[L_c\right]
	\end{equation}
	Here, the symbol $ \mathbb{E}[.] $ denotes the expectation operator. The loss function $ L_{c} $ represents the cost incurred due to the controlled variable $ c $ at a fixed setpoint value, considering a given level of disturbance and measurement noise.	$L_{c, \text{w c}}$ denotes the worst case loss and $L_{c, \text{a v}}$ denotes the average loss.
	
	However, not only should the closed-loop economic loss be considered, but also CVs should be easy to control, that is, insensitive to noise and not be closely correlated.\cite{skogestad:book}
	Next, we will use an example to illustrate the consequences of only considering closed-loop economic losses. 
%
	Here it is assumed that there is only one steady-state degree of freedom. And the controlled variable is designed as:
	\begin{equation}
		\label{eq:c=J-Jopt}
		{c}_J = J(\mathbf{u}, \mathbf{d}) -J^{\mathrm{opt}}(\mathbf{d})
	\end{equation}
	Based on Theorem \ref{the:pCVs}, it could also be represented using measured variables by substituting Eq.~\eqref{eq:d=f(y)} and Eq.~\eqref{eq:u=f(y)} into Eq.~\eqref{eq:c=J-Jopt}, it yields
	$$
	{c}_J = J(\mathcal{F}_u(\mathbf{y}),\mathcal{F}_d(\mathbf{y}))-J^{\mathrm{opt}}(\mathcal{F}_d(\mathbf{y})).
	$$
	It is obviously to get that $ L_{\mathbf{c}_J , w c} $ and $ L_{\mathbf{c}_J ,av} $ both are zero when $ {c}_J  $ is kept at zeros, which may not be a good CV.
	
	As for any feasible disturbance, controlled variables have to satisfy the conditions
	\begin{equation}
			\delta {c}_J=\frac{\partial {c}_J}{\partial {u}} \delta {u}+\frac{\partial {c}_J}{\partial \mathbf{d}} \delta \mathbf{d}=0.
	\end{equation}

	Regrettably, in cases where no constraints are active, it is observed that near the optimal point, $\frac{\partial \mathbf{c}_J}{\partial {u}}$ is always zero while $\frac{\partial \mathbf{c}_J}{\partial \mathbf{d}}$ is not. Consequently, it is impossible to devise a feedback control law that would maintain the controlled variable, $\mathbf{c}_J$, at zero. In other words, this controlled variable is not amenable to closed loop achievability.
	
	
	Based on the discussion before, we give the definition of perfect global self-optimizing CVs.
	\begin{definition}[Perfect global self-optimizing controlled variables]
		The controlled variables (the function of measurements, $\mathbf{c} = h(\mathbf{y})$) are perfect global self-optimizing controlled variables if they satisfy the following conditions:
		\begin{enumerate}
			\item By controlling $\mathbf{c}$ to constant setpoints, the steady-state closed-loop system operates optimally;
			\item The inputs $\mathbf{u}$ should have a significant effect (gain) on $\mathbf{c}$.
		\end{enumerate}
	\end{definition}
	The first condition guarantees the optimality of the controlled variable, and the second condition guarantees the closed loop achievability of the controlled variable. 
	\begin{remark}
		In this paper, we focus on the steady-state behavior of the controlled variables, neglecting their dynamic characteristics. Within this framework, the perfect global controlled variables can be viewed as a feature that characterizes the optimal operation of a system.
		
		However, it should be noted that we do not take measurement noise into consideration in this study. If measurement noise is treated as a disturbance, the number of disturbances in the system is typically greater than the number of measured variables. As a result, achieving optimal operation through measured variables alone may not be feasible.
	\end{remark}
	
	\begin{theorem}(The existence of Perfect global self-optimizing CVs for the g$^2$SOC approach). \label{theo:perfectCV}
		\\
		Let the conditions required by Theorem~\ref{theo:gInvFunc}, Assumption \ref{ass:1}, \ref{ass:2}, \ref{ass:33} and \ref{ass:44} be satisfied, then the perfect global self-optimizing controlled variables $ \mathbf{c}^{\mathrm{perfect}}(\mathbf{y}) $ exist.
	\end{theorem}
	\begin{proof}
		The function
		$$ \mathbf{c}^{\mathrm{perfect}}(\mathbf{y})=A(\mathbf{u}-\mathbf{u}^{\mathrm{opt}}) =A(\mathcal{F}_u(\mathbf{y})- \pi(\mathcal{F}_d(\mathbf{y}))),$$
		 represents perfect SOC CVs, where $ A \in \Real^{n_u} \times \Real^{n_u}$ is a reversible matrix.
		 Obviously, maintaining $ \mathbf{c}^{\mathrm{perfect}}(\mathbf{y}) $ at zeros can achieve optimal operation, which meets condition 1.
		 And $\frac{\partial \mathbf{c}^{\mathrm{perfect}}(\mathbf{y})}{\partial \mathbf{u}}=A$ meets condition 2.
		The conclusion in Theorem \ref{theo:perfectCV} holds directly by combining \eqref{eq:uopt_of_d}, \eqref{eq:d=f(y)} and \eqref{eq:u=f(y)}, which completes the proof.
	\end{proof}
	\begin{remark}\label{rem:A}
		It is worth to note that the matrix $ A $ could be any reversible matrix or any reversible matrix function with respect to $ \mathbf{y} $, at least it should be reversible around the optimal points $(\mathbf{u}^{\mathrm{opt}},\mathbf{d}) $ in $\{\mathcal{U},\mathcal{D}\}$, and it will not affect the steady-state performance of the CVs. This is because the value of $ A $ does not change the control input to maintain the controlled variable at the set value. That is, we can reasonably design $ A $ according to the requirements of the dynamic characteristics of the system.
	\end{remark}
	
	\section{Two classes of Numerical solution methods of the $g^2$SOC approach}
	\label{sec:relationship}
	Theorem~\ref{theo:perfectCV} has defined the existence conditions  of perfect self-optimizing CVs. However, in practice, it is almost impossible to obtain expressions for perfect SOC CVs analytically. It is desirable to devise numerical solutions to approximate the SOC CVs.
	So the core is to construct the function $ h $  which approximates $  \mathbf{c}^{\mathrm{perfect}} $ with $ \mathbf{y} $.

	
	\begin{figure}[htbp]
		\centering
		\includegraphics[width=0.95\textwidth]{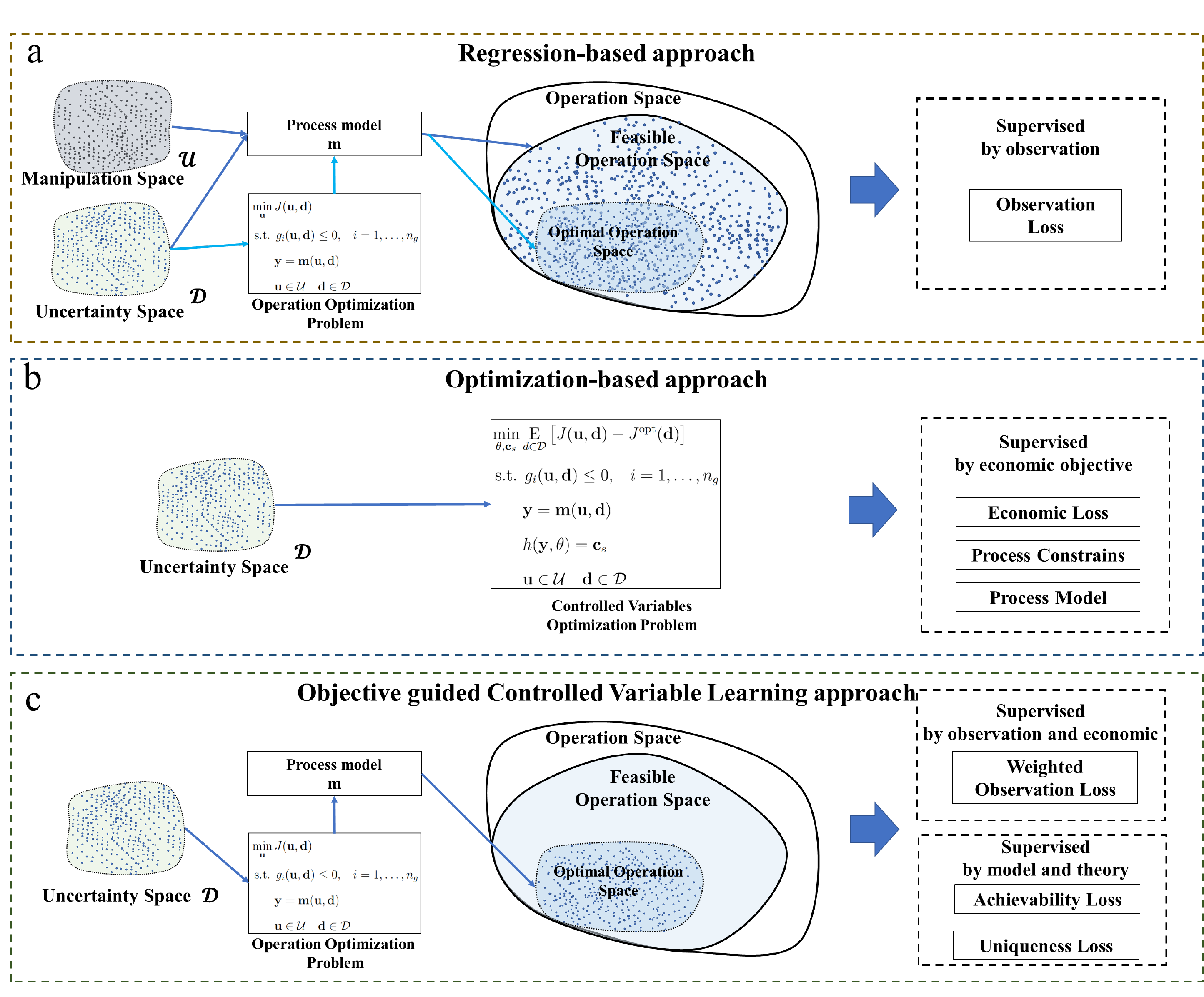} 
		\caption{Illustration of the three approaches (a) Regression-based approach (b) Optimization-based approach (c) Objective guided controlled variable learning approach} 
		\label{fig:overall}
	\end{figure}
	In the first part of this paper\cite{ye2023}, two numerical methods are proposed to approximate the perfect self-optimizing CVs: Regression-based and Optimization-based methods. 
	Fig.~\ref{fig:overall} shows the illustrations of the two methods. The regression-based approach is a type of data-driven approach. This approach requires two parts of data: optimal operation data and non-optimal operation data. In Fig.~\ref{fig:overall}(a), we can see that the regression-based approach samples within the operation and uncertainty spaces, generates non-optimal data through the process model, and then solves operation optimization problem under the associated disturbances sampled before for the optimal operation data. Finally, the perfect self-optimizing controlled variable is approximated by solving the following regression problem. Usually it can be expressed as:
	\begin{equation}\label{eq:regress_u}
		\min _{\theta} \underset{\mathbf{d} \in \mathcal{D}, \mathbf{u} \in \mathcal{U}}{ \operatorname{E}}\left[ \|(h(\mathbf{y},\theta)- (\mathbf{u}-\mathbf{u}^{\mathrm{opt}})) \|^2_2\right] \approx \dfrac{1}{N}\sum_{i=1}^{N} \|(h(\mathbf{y}_i,\theta)- (\mathbf{u}_i-\mathbf{u}^{\mathrm{opt}}_i)) \|^2_2 \
	\end{equation}
	where $ \operatorname{E}[.] $ denotes the expected value of a random variable, $ h:\Real^{n_{y}}\mapsto \Real^{n_{u}} $ is a pre-determined regression model, and always a parametric family of functions, and $\theta$ is the parameter of this regression model. $ N= N_d \times N_u $ represents the number of samples, $ N_d $ and $ N_u $ denote the number of disturbance samples and manipulated variables samples, respectively. And the subscript $ i $ represents the i-th sample.
	
	Correspondingly, as shown in Fig.~\ref{fig:overall}(b), optimization-based method approach only needs to sample in the uncertain space, and then solve the following optimization problem. 
	\begin{equation}  
			\label{eq:soc}
			\begin{aligned}
					&\min _{\theta,\mathbf{c}_s} \underset{d \in \mathcal{D}}{ \operatorname{E}}\left[J(\mathbf{u}, \mathbf{d}) - J^{\mathrm{opt}}(\mathbf{d}) \right] \approx \dfrac{1}{N_d}\sum_{j=1}^{N_d}\left(J(\mathbf{u}, \mathbf{d}_j) - J^{\mathrm{opt}}(\mathbf{d}_j) \right)\\
					&\text { s.t. }  g_i(\mathbf{u}_j, \mathbf{d}_j) \leq 0  ,\;\;\;i=1,\ldots,n_g \\
					& \qquad \mathbf{y}_j = \mathbf{m}(\mathbf{u}_j, \mathbf{d}_j) \\
					& \qquad h(\mathbf{y}_j,\theta) = \mathbf{c}_s \\
					& \qquad \forall j=1,2,...,N_d \\ 
				\end{aligned}
		\end{equation}
	where 
	$ \mathbf{c}_s \in \Real^{n_{u}} $ is the set-points of the CVs. They are constant and usually assigned in advance.
	$ N_d $ represents the number of disturbance  samples, and the subscript $ j $ represents the j-th sample.

	
	Both types of methods can be used to construct perfect SOC CVs, however, they have their own advantages and disadvantages, and we will conduct an in-depth analysis of them next.
	\begin{theorem}(Equivalence)\\
		\label{the:Equ}
		The solution of Eq.~\eqref{eq:regress_u} is also the solution of Eq.~\eqref{eq:soc}.
	\end{theorem}
	\begin{proof}
		$ \left(\mathcal{F}_u(\mathbf{y})-\pi(\mathcal{F}_d(\mathbf{y}))\right)  $ is the solution of Eq.~\eqref{eq:soc}, and the value of objective of Eq.~\eqref{eq:regress_u} is zero.
		$ \pi(\mathcal{F}_d(\mathbf{y})) $ is also the solution of Eq.~\eqref{eq:soc}, and the value of objective is also zero.
	\end{proof}
	
	According to Theorem~\ref{the:Equ}, the regression method and the optimization method are ideally equivalent, and both can be used to design/approximate the perfect self-optimizing controlled variable. This transforms the original controlled variable design problem from a complex nonlinear programming problem to a machine learning problem.
	
	However, in practice, even under the same parametric approximating model, those approaches' performances are not same. Next, we will introduce two concepts in machine learning, the model capacity and sample representativeness to illustrate the difference between those two approaches in practice.
	
	\begin{remark}[Capacity of models]
		\label{rem:modelCapability}
		Informally, a model’s capacity is its ability to fit a wide variety of functions.\cite{goodfellow2016}
		In practice, model capacity could be controlled by the model structure, for example, the capacity of a neural network can be controlled by the number of nodes and the number of layers.
		And the model capacity needs to alter carefully to avoid either overfit or underfit.
	\end{remark}
	
	\begin{remark}[Representativeness of samples]
		\label{rem:samples}
		Under the same model structure, the efficiency of the model is depended on representativeness of samples. The model usually performs better if the samples better represent the sampling space.
		This means that the way of sampling and the weights among samples can be adjusted to improve the efficiency of the model.	
	\end{remark}
	In practice, both methods have their own advantages and disadvantages:
	\begin{enumerate}
		\item \textbf{Optimization better than Regression},
			\subitem \textbf{Objective Function}. Optimization approach considers the close-loop economic performance, meanwhile, the regression approach only considers minimizing fitting error. Under some approximated model structure, the theoretical achievable performance for the former is always not worse than the latter. This can be interpreted as optimization methods using economic objectives to allocate model capacity differentially, and most of the impact on closed-loop economic performance has gained more attention.
			\subitem \textbf{Sampling Space}. The sampling space of optimization approach is a proper subset of the sampling space of regression approach, which makes the representativeness of the former method samples could be not worse than the latter. In other words, SOC approach weights different samples by considering closed-loop economic performance, and limits the operating space to make the samples more representative.
			\subitem \textbf{Model constrains.} The regression approach is primarily concerned with minimizing training error by fitting the relationship between input and output, without incorporating any structured information or domain knowledge. Consequently, regression-based methods may produce results that violate the laws of physics.
			On the other hand, the optimization approach explicitly integrates the process model, enabling the controlled variables to be designed in accordance with the constraints of the process model. This ensures that the resulting controlled variables adhere to the physical laws governing the process.
		\item \textbf{Regression better than Optimization}, 
			\subitem \textbf{Algorithm feasibility and complexity}. The optimization approach involves the nonlinear process model, which makes Eq.~\eqref{eq:soc} a non-convex nonlinear programming problem which is difficult to solve. Correspondingly, regression approach samples in the space of $\mathcal{U}$ and $\mathcal{D}$, and the process model information is implicitly included in the sample data, which greatly reduces the difficulty of solving the problem. What's more, there are some standard algorithms/toolboxes designed for performing regression, which also makes it much easier to solve.
	\end{enumerate}
	
	The Objective-guided Controlled Variables Learning(OGCVL) approach in Fig.~\ref{fig:overall}(c) has a specially designed loss function and sampling space, in which weighted fitting error can be regarded as a simplified closed-loop economic performance indicator, 
	regular terms that introduce domain knowledge are also added, and we only need to sample in the optimal operation space. The details of the OGCVL approach are introduced in Section~\ref{sec:g2SOC}.
	
	\section{Objective-guided Controlled Variables Learning for $g^2$SOC}
		\label{sec:g2SOC}

	{
	The inspiration for the OGCVL approach arose from a deep analysis of the strengths and limitations of both optimization-based and regression-based methods for designing CVs. Our key insight was recognizing that these two approaches, while theoretically equivalent, differ significantly in their practical implementation and outcomes.
	The optimization-based method, while potentially more accurate, often becomes computationally intractable for complex systems. On the other hand, the regression-based method, though more computationally efficient, may not fully capture the underlying physical constraints and optimality conditions. This realization led us to explore a hybrid approach that could harness the strengths of both methods.
	The breakthrough came when we considered framing the CV design problem within the context of advanced machine learning. This perspective allowed us to identify three critical components that could be tailored to suit the specific requirements of CV design: Sampling Space, Loss Function Design and Machine Learning Model Structure.
	This method is described in detail next.}

	\subsection{The framework of Objective-guided Controlled Variables Learning}
	
	The trade-off between computational efficiency and accuracy also needs to be considered in the field of partial differential equation solving\cite{hartEnvironmentalSensorNetworks2006a}.
	Recently, there has been a surge in the popularity of 'physics-informed neural networks' (PINNs). This class of deep learning algorithms has the unique capability of effectively integrating data with abstract mathematical operators, such as partial differential equations (PDEs). Essentially, these neural networks are trained to solve supervised learning tasks while adhering to any pre-existing laws of physics. Consequently, this approach has been successful in achieving a harmonious balance between fundamental physical principles and observational data.\cite{Raissi2019,karniadakisPhysicsinformedMachineLearning2021}
	The idea of PINNs can be used to bridge data-driven and model-driven approaches\cite{wang2020}.
	Regression-based $g^2$SOC methods can be considered as data-driven methods, and optimization-based $g^2$SOC methods can be considered as model-based methods. Therefore, inspired by the PINNs method, we developed an Objective-guided Controlled Variable Learning(OGCVL) framework for the SOC CVs design problem, to fully utilize both data and models.
	
	\begin{figure}[htbp]
		\centering
		\includegraphics[width=0.9\textwidth]{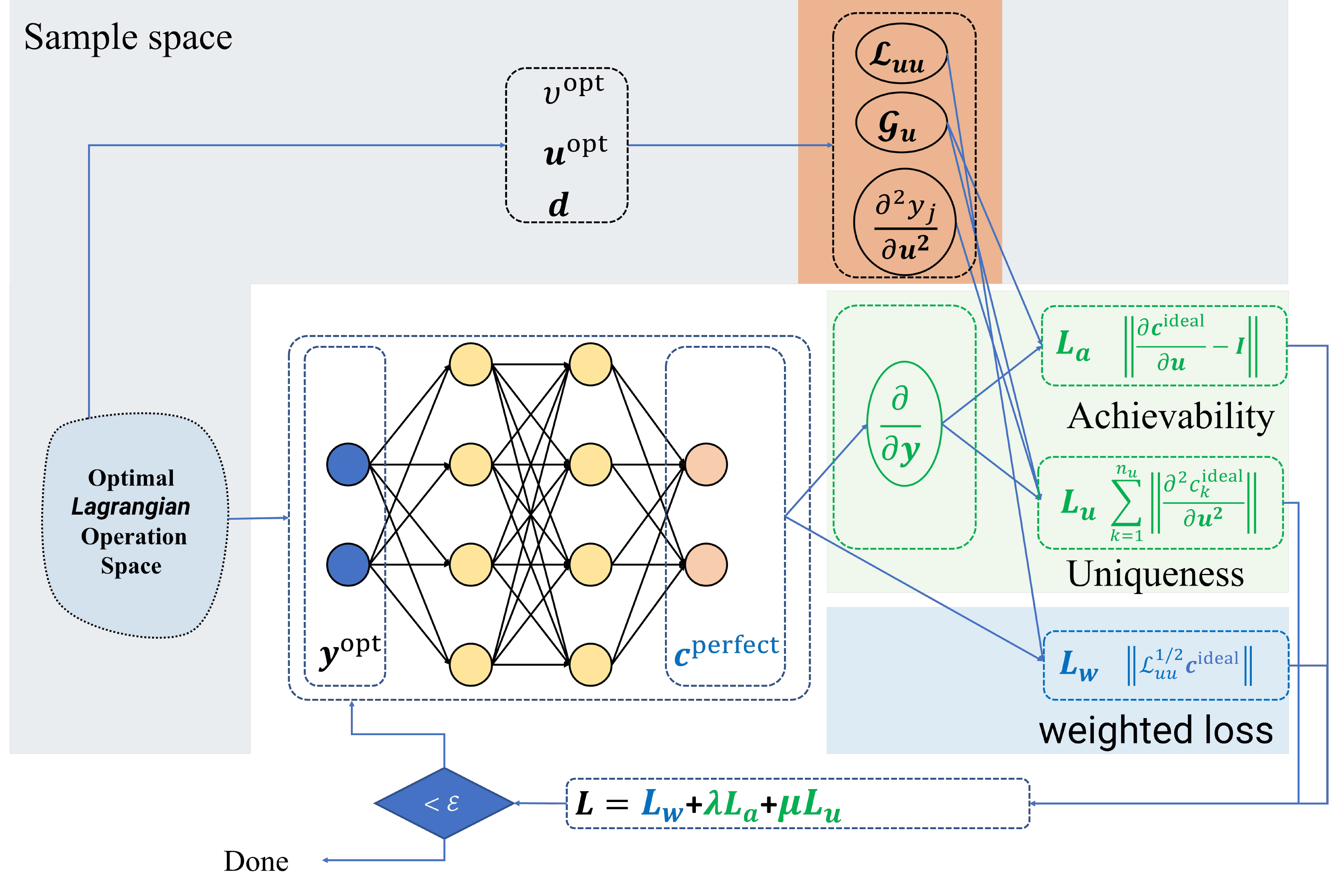} 
		\caption{Flow chart of the Objective-guided Controlled Variables Learning framework} 
		\label{fig:framework}
	\end{figure}
	
	The Objective-Guided Controlled Variable Learning (OGCVL) framework is illustrated in Fig.~\ref{fig:framework}. This framework can be explained component by component as follows:
	\begin{itemize}
		\item \textbf{Optimal Lagrangian Operation Space:} This is represented by the gray ellipse on the left. It serves as the sampling space for the framework, indicating the type of data that needs to be collected. This space includes the optimal inputs $\mathbf{u}^{\text{opt}}$, disturbances $\mathbf{d}$, optimal measurements $\mathbf{y}^{\text{opt}}$, and corresponding optimal Lagrange multiplier $\nu^{\text{opt}}$.
		\item \textbf{Neural Network:} The central part of the figure shows a neural network structure. This network takes $\mathbf{y}^{\text{opt}}$ as input and outputs $\mathbf{c}^{\text{ideal}}$, which represents the ideal controlled variables.
		
		\item \textbf{Derivative Calculations:} The orange box on the right contains $\frac{\partial}{\partial y}$ and $\frac{\partial^2}{\partial y^2}$, representing the first and second derivatives of the neural network output with respect to the input. These are used in the regularization terms.
		
		\item \textbf{Regularization Terms:} The green boxes $L_a$ and $L_u$ represent symbolic regularization terms. $L_a$ enforces the condition that $\|\frac{\partial \mathbf{c}^{\text{ideal}}_k}{\partial \mathbf{u}} - I\|$ should be minimized, addressing closed-loop achievability. $L_u$ minimizes $\sum_{k=1}^{n_u}\|\frac{\partial c^{\text{ideal}}_k}{\partial \mathbf{u}}\|$, ensuring the uniqueness of the CV solution.
		
		\item \textbf{Weighted Loss:} The blue box $L_w$ represents the weighted loss, which links the performance to the observed loss in the closed-loop economy. It is calculated as $\|L_{uu}^{1/2}\mathbf{c}^{\text{ideal}}\|$.
		
		\item \textbf{Total Loss:} The overall loss function $L$ is a weighted sum of $L_w$, $L_a$, and $L_u$, with weights $\lambda$ and $\mu$.
		
		\item \textbf{Optimization Loop:} The framework operates in an iterative manner. If the total loss $L$ is above a threshold $\varepsilon$, the process continues; otherwise, it terminates.
	\end{itemize}
	This framework combines the strengths of data-driven approaches (through the neural network) and model-based optimization (through the regularization terms and weighted loss). It aims to learn controlled variables that are both computationally efficient to obtain and ensure near-optimal closed-loop performance.

	\subsection{Optimal control input learning: Learning in optimal operation space}
	The regression methods need to sample the entire operation space, including $\mathcal{D}$ and $\mathcal{U}$. This is a purely data-driven approach which learn the inverse dynamics relationship, $ \mathcal{F}_u(\mathbf{y})$ and $ \mathcal{F}_d(\mathbf{y}) $, and concurrently the optimal solution mapping, $\pi(\mathbf{d}) $ . On the other hand, optimization methods leverage process models to reduce the sampling space to the uncertainty space $\mathcal{D}$. These methods utilize the degree of freedom of manipulated variables $\mathbf{u}$, together with the CVs function, implying that the control inputs $\mathbf{u}$ are functions of $\mathbf{d}$. Similarly, optimal control inputs $\mathbf{u}^{\mathrm{opt}}$ are functions of $\mathbf{d}$. This raises the question of whether it is possible to approximate the perfect self-optimizing CVs by only sampling the optimal operation space $\{(\mathbf{u}^{\mathrm{opt}}(\mathbf{d}),\mathbf{d})|\mathbf{d}\in\mathcal{D}\}$. 
	
	As depicted in Eq.~\eqref{eq:yopt}, optimal operation space comprises both optimal solution information and system dynamics information. 
	\begin{equation}\label{eq:yopt}
		\begin{aligned}
			\mathbf{y}^{\mathrm{opt}} &= \mathbf{m}(\mathbf{u}^{\mathrm{opt}},\mathbf{d})\\
			\mathbf{u}^{\mathrm{opt}} &= \pi(\mathbf{d})
		\end{aligned}
	\end{equation}
	 Theoretically, if those information could be fully exploited, it would be possible to construct perfect self-optimizing controlled variables. Nevertheless, applying the original problem framework~\eqref{eq:regress_u} directly is infeasible. The following analysis, therefore, describes the necessary problem modification.
	
	Transforming the sampling space of Problem~\ref{eq:regress_u} directly into the optimal operation space yields this expression:	
	\begin{equation}\label{eq:regress_u_opt} 
		\min _{\theta} \underset{\mathbf{d} \in \mathcal{D}}{ \operatorname{E}}\left[ |(h(\mathbf{y}^{\mathrm{opt}},\theta)- (\mathbf{u}^{\mathrm{opt}}-\mathbf{u}^{\mathrm{opt}})) |^2_2\right]. \ 
	\end{equation}
	Unfortunately, this problem is singular since the regression model's output is always zero, resulting in a trivial solution, $\theta = \mathbf{0}$. 
	Because the control inputs $\mathbf{u}$ are usually measurable, it is only necessary to learn $ \pi(\mathcal{F}_d(\mathbf{y}^{\mathrm{opt}})) $  from the optimal operation data. The perfect self-optimizing controlled variable approximation problem could be simplified to the optimal control input $\mathbf{u}^{\mathrm{opt}}$ regression problem. 
	\begin{equation}\label{eq:regress_u_opt2}
		\min _{\theta} \underset{\mathbf{d} \in \mathcal{D}}{ \operatorname{E}}\left[ \|(f(\mathbf{y}^{\mathrm{opt}},\theta)- \mathbf{u}^{\mathrm{opt}}) \|^2_2\right] \
	\end{equation}
	The CVs obtained by solving Problem~\ref{eq:regress_u_opt2} could be presented as $h(\mathbf{y})=\mathbf{u}-f(\mathbf{y},\theta)$.

	\subsection{Achievability and uniqueness regularizer}
	Although singularities are avoided, there is still a challenge when solving this problem directly through the OGCVL approach, since $ \mathcal{F}_u(\mathbf{y}^{\mathrm{opt}}) $ and $ \pi(\mathcal{F}_d(\mathbf{y}^{\mathrm{opt}})) $ can both be the solution of this problem.
	However, only the latter one is the perfect solution. This method avoids singularities, but the result may not be what we really want. 
	Therefore, to address these problems, two regularizers are designed according to the definition of the perfect SOC CVs——achievability regularizer and uniqueness regularizer.
	\subsubsection{Close loop achievability regularizer}
	 To ensure closed-loop achievability of $h$, we need its Jacobian matrix $\mathbf{c}_u=\frac{\partial \mathbf{h}(\mathbf{y})}{\partial \mathbf{u}}\big|_{\mathbf{u}=\mathbf{u}^{\mathrm{opt}}}$ to be full rank, that is $rank(\mathbf{c}_u)=n_u$.
	Since the specific value of $\mathbf{c}_u$ does not affect the steady-state economic performance of the controlled variable (Remark~\ref{rem:A}), we can specify it as an identity matrix $I_{n_u}$ without loss of generality.
	Therefore, we introduce an achievability penalty $ L_a $ for the CVs
	\begin{equation}
		\label{eq:la}
		L_a = \|\mathbf{c}_u -I_{n_u}\|_2^2 = \left\|\frac{\partial h(\mathbf{y})}{\partial \mathbf{y}}\mathcal{G}_{u} -I_{n_u}\right\|_2^2
	\end{equation}
	where $ \mathbf{c}_u $ is calculated by applying the chain rule, $ \mathbf{c}_u = \frac{\partial \mathbf{h}(\mathbf{y})}{\partial \mathbf{y}}\frac{\partial \mathbf{y}}{\partial \mathbf{u}} $. $\mathcal{G}_{u} \triangleq \frac{\partial \mathbf{y}}{\partial \mathbf{u}} $ denotes Jacobian matrices of appropriate sizes, which are evaluated around the optimal point, representing the gain from the control inputs $ \mathbf{u} $ to the measurements $\mathbf{y}$. And for the optimal control input regression approach, achievability penalty can be modified to:
	\begin{equation}
		\label{eq:la_1}
		L_a = \|\mathbf{c}_u -I_{n_u}\|_2^2 = \left\|f_y\mathcal{G}_{u}\right\|_2^2
	\end{equation}
	where $f_y \triangleq \frac{\partial f(\mathbf{y})}{\partial \mathbf{y}}$ denotes the Jacobian matrices of $  f(\mathbf{y}) $ to $ \mathbf{y} $.
	Since it is only necessary to ensure that CVs are achievable at the optimal point, $ L_a $ is calculated at the optimal point. We can thus formulate our regression objective as minimizing the expected squared norm of $h(\mathbf{y}^{\mathrm{opt}},\theta)$, plus a penalty term proportional to the distance between $\mathbf{c}_u$ and $I_{n_u}$, where $\lambda$ is the penalty factor:
	\begin{equation}\label{eq:regress_u_opt3}
		\min_\theta \underset{\mathbf{d} \in \mathcal{D}}{ \operatorname{E}}
		\left[
		 L_o+\lambda L_a
		\right] 
		\approx \dfrac{1}{N_d}\sum_{j=1}^{N_d}\left(
		L_{o,j} +\lambda L_{a,j}
		\right) . 
	\end{equation}
	where $L_o =  \|(f(\mathbf{y}^{\mathrm{opt}},\theta)) - \mathbf{u}^{\mathrm{opt}}\|^2_2 $ denotes the observation loss.

	\subsubsection{Uniqueness regularizers}

	Setting the gradient of the controlled variable to be zero only at the optimal point and non-zero elsewhere guarantees the closed-loop achievability of the controlled variable.
	Nevertheless, this method is effective only when the controlled variable is in the vicinity of the optimal point. Furthermore, there is no certainty that the controlled variable function possesses only one zero point in the entire operation region.
	In this paper, we assume that for each disturbance $\mathbf{d}$, there exists only one optimal solution for the static operation optimization problem. To ensure that our controlled variables have unique control inputs that keep them at set points, we introduce a regularization method that is motivated by Theorem~\ref{the:OneSolutionc=0}.
	
	\begin{theorem}
		\label{the:OneSolutionc=0} 
		Let $\mathfrak{f}(x):\Real^{n_x} \mapsto \Real$ be a continuously differentiable function s.t. $\left\|\mathfrak{f}_{xx}\right\|=\left\|\frac{\partial^2{\mathfrak{f}}}{\partial{x}^2}\right\| \leq \epsilon $ for some $\epsilon>0$. Let $\mathbf{a}$ satisfy the relationship that $\mathfrak{f}(\mathbf{a})=\alpha$ , $ \mathfrak{f}_\mathbf{a}=\frac{\partial{\mathfrak{f}}}{\partial{x}}\big|_{x=\mathbf{a}}=A$ and $A$ is reversible. Then, $\nexists \mathbf{b}$ such that $0<\|\mathbf{b}-\mathbf{a}\|<\frac{2}{\|A^{-1}\|\|\epsilon\|}, \mathfrak{f}(\mathbf{b})=\alpha $.
	\end{theorem}
	\begin{proof}
		Proof by contradiction. Assume $\exists \mathbf{b} \neq \mathbf{a}$, such that $0<\|\mathbf{b}-\mathbf{a}\|<\frac{2}{\left\| A^{-1} \right\|\|\epsilon\|}, \mathfrak{f}(\mathbf{b})=\alpha $.
		For function $\mathfrak{f}(\mathbf{x})$, the Taylor series with second order Lagrange remainder about the point $\mathbf{a}$ is 
		$$
		\mathfrak{f}(\mathbf{x})=\mathfrak{f}(\mathbf{a})+\mathfrak{f}_x(\mathbf{a})(\mathbf{x}-\mathbf{a})+\frac{1}{2}(\mathbf{x}-\mathbf{a})^{\top}\mathfrak{f}_{xx}(\xi)(\mathbf{x}-\mathbf{a})
		$$
		where $\xi \in\left\{\mathbf{a}+\beta\left(\mathbf{x}-\mathbf{a}\right) \mid \beta \in[0,1]\right\}$
		Then,
		$$
		\begin{aligned}
			\mathfrak{f}(\mathbf{b}) & =\mathfrak{f}(\mathbf{\mathbf{a}})+\mathfrak{f}_x(\mathbf{a})(\mathbf{b}-\mathbf{a})+\frac{1}{2}(\mathbf{b}-\mathbf{a})^{\top}\mathfrak{f}_{xx}(\xi)(\mathbf{b}-\mathbf{a}) \\
			2(\mathbf{b}-\mathbf{a}) &= -A^{-1}(\mathbf{b}-\mathbf{a})^{\top}\mathfrak{f}_{xx}(\xi)(\mathbf{b}-\mathbf{a})\\
			2\left\|(\mathbf{b}-\mathbf{a}) \right\| &\leq \left\| A^{-1} \right\| \left\| \mathfrak{f}_{xx} \right\| \left\|\mathbf{b}-\mathbf{a} \right\|^2 \\
			&\leq \left\| A^{-1} \right\| \epsilon \left\|\mathbf{b}-\mathbf{a} \right\|^2 \\
			\left\|\mathbf{b}-\mathbf{a} \right\| & \geq \frac{2}{\left\| A^{-1} \right\| \epsilon} \\
		\end{aligned}
		$$
		This contradicts with $\|\mathbf{b}-\mathbf{a}\|<\frac{2}{\|A^{-1}\|\epsilon} $. This completes the proof.
	\end{proof}
	
	Theorem~\ref{the:OneSolutionc=0} states that if a continuously differentiable function $\mathfrak{f}(x)$ satisfies certain conditions around a point $\mathbf{a}$ where its gradient is reversible, then within a certain distance $ \frac{2}{\|A^{-1}\|\epsilon} $ of $\mathbf{a}$ there is no such a point $\mathbf{b}$ such that $\mathfrak{f}(\mathbf{b})$ is the same as $\mathfrak{f}(\mathbf{a})$. We can apply this theorem to our problem by considering the controlled variable function $h(\mathbf{y})$ as a continuously differentiable function of the control input $\mathbf{u}$ and disturbance $\mathbf{d}$.
	
	To apply Theorem~\ref{the:OneSolutionc=0}, we substitute the process model $\mathbf{y}=m(\mathbf{u},\mathbf{d})$ into the CVs function $h(\mathbf{y})$ to get a modified function $\hat{h}(\mathbf{u},\mathbf{d})$. For each disturbance $\mathbf{d}$, we know that the optimal control input $\mathbf{u}^{\mathrm{opt}}$ must satisfy $\hat{h}_k(\mathbf{u}^{\mathrm{opt}},\mathbf{d})=\mathbf{0}, k =1,2,..,n_u$(k-th controlled variable function). In this setting, the Hessian matrix of $\hat{h}_k$ with respect to $\mathbf{u}$, ${c}_{k,uu}$ , plays the role of the Hessian matrix $\mathfrak{f}_{xx}$ in Theorem~\ref{the:OneSolutionc=0}. 
	If the norm of the Hessian matrix $c_{k,uu}$ is sufficiently small, then Theorem~\ref{the:OneSolutionc=0} guarantees that only one control input $\mathbf{u}$ can make the controlled variable function $\hat{h}_k$ equal to $\mathbf{0}$ in the feasible operation space $\mathcal{U}$. Therefore, we can add a regularization term $ L_u $ based on the norm of the Hessian matrix to the regression objective to ensure unique solutions.
	\begin{equation}
		\label{eq:lu}
		L_{u} = \sum_{k=1}^{n_u} \| c_{k,uu} \|_2^2 
	\end{equation}
	$ {c}_{k,uu} $ represents the second derivative of k-th controlled variable ${c}_k$ with respect to $\mathbf{u}$. 
	\begin{equation}
		{c}_{k,{uu}} = \frac{\d c_{k,{u}}}{\d \mathbf{u}} 
		= \frac{\d (1- f_{k,y}\mathcal{G}_u)}{\d \mathbf{u}}
		= \frac{\partial c_{k,{u}}}{\partial \mathbf{u}} + \frac{\partial c_{k,{u}}}{\partial \mathbf{y}} \frac{\partial \mathbf{y}}{\partial \mathbf{u}}= \mathcal{G}_u^{\top}f_{k,yy}\mathcal{G}_u +\sum_{j=1}^{n_y}f_{k,y}(j,:) \frac{\partial^2 {y}_j}{\partial \mathbf{u}^2}
	\end{equation}
	where $f_{k,yy}$ denotes the Hessian matrix of the k-th mapping $f$ with respect to $\mathbf{y}$, $f_{k,y}(j,:) \in \Real$ represents the j-th row of the Jacobian matrix of the k-th mapping $f$ with respect to $\mathbf{y}$, and $\frac{\partial^2 {y}_j}{\partial \mathbf{u}^2}$ denotes
	 the Hessian matrix of the $j$-th measured variables $\mathbf{y}$ to $\mathbf{u}$
	
	\begin{remark}
		If the ReLU function is employed as the activation function for the neural network, the calculation of ${c}_{k,{uu}}$ can be simplified, as in this case, $f_{k,yy}$ = 0, leading to the following expression:
		\begin{equation}
			{c}_{k,{uu}} = \sum_{j=1}^{n_y}f_{k,y}(j,:) \frac{\partial^2 {y}_j}{\partial \mathbf{u}^2}
		\end{equation}
	\end{remark}
	
	
	
	Ultimately, the following problem can be obtained
	\begin{equation}\label{eq:regress_u_opt4}
		\begin{aligned}
			\min_\theta \underset{\mathbf{d} \in \mathcal{D}}{ \operatorname{E}}
			\left[
			L_o	+  \lambda L_a	+\mu L_u
			\right]
			\approx \dfrac{1}{N_d}\sum_{j=1}^{N_d}\left(L_{o,j}	+  \lambda L_{a,j}	+\mu L_{u,j}
			\right) . 
		\end{aligned}		
	\end{equation}
	where $ \mu $ denotes the penalty factor.

	 In this section, we develop the implicit function learning approach to fine-tune the settings of the regression method (Eq.~\ref{eq:regress_u}) by introducing regularization terms that relate to first-order and second-order gradients, which guarantees that the perfect SOC CVs can be successfully attained solely by performing sampling in the optimal space.
	 Furthermore, it also guarantees that the CVs function has a unique zero point to a significant range and the close loop achievability, thereby enhancing the reliability of CVs.
	
	\subsection{Objective-guided loss function design with Lagrangian multipliers}
	
	The optimization approach offers a key advantage of optimizing the closed-loop economic performance directly, whereas the regression approach aims to minimize the fitting error, which bears no physical significance. Some studies \cite{jaschkeOptimalOperationControlling2011,YeCao2013IECR} propose using the inverse square root of the reduced Hessian matrix of economic function $J$ to $\mathbf{u}$, represented as $J^{-1/2}_{uu}$, as a weighting factor to associate the regression fitting error with economic indicators. However, the suggested approach has limitations as it assumes that the correct active set is known and remains stationary, and it only uses the gradient as the CV, without considering more general CVs. To broaden the scope of this idea and to overcome these limitations, our research aims to extend the notion of economic weighting to a wider operational range. We adopt the Lagrangian multipliers method to evade the impact of varying active constraints.
	
	The Lagrangian function for Eq.~\eqref{eq:sopt} can be represented as
	\begin{equation} 
		\label{eq:L_sopt}
		\begin{aligned}
			\mathcal{L}({\mathbf{u}},\mathbf{d})= J({\mathbf{u}},\mathbf{d})+\sum_{i=1}^{n_g}\nu_i g_i({\mathbf{u}},\mathbf{d}), \  \ i=1,2,...,n_g 
		\end{aligned}
	\end{equation}
	where $ \nu_i \geq 0 ,i=1,2,...,n_g $ is the Lagrangian multiplier, and $ n_g $ denotes the number of constrains.
	
	\begin{theorem}[Lagrange multiplier theorem\cite{fuente2000}]
		Let $f: \mathbb{R}^n \rightarrow \mathbb{R}$ be the objective function, $g: \mathbb{R}^n \rightarrow \mathbb{R}^c$ be the constraints function, both belonging to $C^1$ (that is, having continuous first derivatives). Let $x^{\mathrm{opt}}$ be an optimal solution to the following optimization problem such that 
		$\operatorname{rank}\left({\partial g}/{\partial x}|_{x=x^{\mathrm{opt}}}\right)=c<n$:
		$$
		\begin{aligned}
			& \min_x f(x) \\
			& \text { s.t. } g(x)=0
		\end{aligned}
		$$
		Then there exists a unique Lagrange multiplier $\nu^{\mathrm{opt}} \in \mathbb{R}^c$ such that 
		$$\frac{\partial f}{\partial x}\bigg|_{x=x^{\mathrm{opt}}} = {\nu^{\mathrm{opt}}}^{\top} \frac{\partial g}{\partial x}\bigg|_{x=x^{\mathrm{opt}}}.$$
	\end{theorem}
	
	Considering the Lagrange multiplier theorem, Eq.~\eqref{eq:sopt} and Eq.~\eqref{eq:L_sopt} are equivalent, and for each $ \mathbf{u}^{\mathrm{opt}} $ there is one and only one Lagrangian multiplier $ \nu^{\mathrm{opt}} $ corresponding to it.
	Therefore, it can also be understood that there is a mapping  $\Pi:\Real^{n_d} \mapsto \Real^{n_g}$
	\begin{align}\label{eq:nuopt_of_d}
		\nu^{\mathrm{opt}} &= \Pi(\mathbf{d})
	\end{align}

	Then objective of Eq.~\eqref{eq:L_sopt} is approximated locally by a quadratic function around an optimal input $ \mathbf{u}^{\mathrm{opt}} $.
	\begin{equation}\label{eq:dJ}
		\begin{aligned}
			\mathcal{L}(\mathbf{u},\mathbf{\nu}^{\mathrm{opt}}, \mathbf{d}) \approx 
			\mathcal{L}(\mathbf{u}^{\mathrm{opt}},\mathbf{\nu}^{\mathrm{opt}}, \mathbf{d}) 
			+ \mathcal{L}_u(\mathbf{u}  - \mathbf{u}^{\mathrm{opt}}) 
			+ \frac{1}{2}(\mathbf{u}  - \mathbf{u}^{\mathrm{opt}})^\top \mathcal{L}_{uu} (\mathbf{u}  - \mathbf{u}^{\mathrm{opt}})
		\end{aligned}
	\end{equation}
	where $ \mathcal{L}_{u} =  \dfrac{\partial \mathcal{L}}{\partial \mathbf{u}}$ denotes the Jacobian matrix of $ \mathcal{L} $ with respect to $ \mathbf{u} $, 
	$ \mathcal{L}_{uu} =  \dfrac{\partial^2 \mathcal{L}}{\partial \mathbf{u}^2}$ denotes the Hessian matrix of $ \mathcal{L} $ with respect to $ \mathbf{u} $.
	It is noted that the $ \mathcal{L}_{u}|_{\mathbf{u} = \mathbf{u}^{\mathrm{opt}}}=0 $ whether or not there are constraints. Therefore, the loss of the Lagrangian function $ \mathcal{L} $ generated by non-optimal inputs near the optimal operating point can be approximated as:
	\begin{equation}
		\label{eq:dL}
		\Delta\mathcal{L}(\mathbf{u}, \mathbf{d})\triangleq\mathcal{L}(\mathbf{u}, \mathbf{d}) - \mathcal{L}(\mathbf{u}^{\mathrm{opt}}, \mathbf{d})  \approx \frac{1}{2}(\mathbf{u}  - \mathbf{u}^{\mathrm{opt}})^\top \mathcal{L}_{uu} (\mathbf{u}  - \mathbf{u}^{\mathrm{opt}})
	\end{equation}
	For close loop, the CVs $\mathbf{c} = \mathbf{u} - f(\mathbf{y})$ are maintained at 0. 
	Therefore, $ \Delta\mathcal{L}(\mathbf{u}, \mathbf{d}) $ can be further modified to:
	\begin{equation}
		\begin{aligned}
			\Delta\mathcal{L}(\mathbf{u}, \mathbf{d})&\approx \frac{1}{2}
			\left(\mathbf{u}-	f(\mathbf{y}) +	f(\mathbf{y}) - \mathbf{u}^{\mathrm{opt}}  \right)^\top
			\mathcal{L}_{uu} \left(\mathbf{u}-	f(\mathbf{y}) +	f(\mathbf{y}) -\mathbf{u}^{\mathrm{opt}}  \right)\\
			&=\frac{1}{2}\left(	f(\mathbf{y}) - \mathbf{u}^{\mathrm{opt}}  \right)^\top
			\mathcal{L}_{uu} \left(	f(\mathbf{y}) -\mathbf{u}^{\mathrm{opt}}  \right)
		\end{aligned}
	\end{equation}
	
	Thus, by weighting different samples under the guide of objective, Problem~\ref{eq:regress_u_opt4} can be representated as:
	\begin{equation}
		\min_\theta \underset{\mathbf{d} \in \mathcal{D}}{ \operatorname{E}}
		\left[
		L_w+  \lambda L_a	+\mu L_u
		\right]
		\approx \dfrac{1}{N_d}\sum_{j=1}^{N_d}\left(L_{w,j}	+  \lambda L_{a,j}	+\mu L_{u,j}
		\right) 
	\end{equation}	
	where $ L_w $ denotes the weighted observation loss and it is given as fellows
	\begin{equation}
		\label{eq:lw}
		 L_w = \|\mathcal{L}_{uu}^{1/2}(h(\mathbf{y}^{\mathrm{opt}},\theta)) \|^2_2
	\end{equation} 
	
	\begin{remark}
		This formulation utilizes the Hessian matrix information of the Lagrangian function to weight the samples, which has been employed in \cite{su2023,zhou2024a}. Concurrently, Zhou et al. \cite{zhou2024a} have given a proof substantiating that this weighted form yields superior closed-loop performance compared to its unweighted counterpart.
	\end{remark}
	
	To sum up, implementation of the Objective-guided Controlled Variables Learning based $g^2$SOC approach is given in Algorithm 1. 
	
	\begin{algorithm}[h]
		\caption{The OGCVL g$^2$SOC approach}
		\label{algo:OptSOC}
		\KwIn{The formulation of optimization problem~\eqref{eq:L_sopt};A neural network with a given structure $h(\mathbf{y},\theta)$;Initial value $\theta_0$}
		\KwOut{$\mathbf{c}=h(\mathbf{y},\theta_{\text{opt}})$.}
		\BlankLine
		Discretization over $\mathcal D$, a set of discretized points $\tilde{\mathcal X}=\{d_{(i)}\}_{i=1,\ldots,N}$ is generated;
		
		Set suitable values for $\mu$ and $\lambda$
		
		\ForEach{$d_{(i)}$ in $\tilde{\mathcal X}$}
		{ 
			Solving Eq.~\eqref{eq:L_sopt} to obtain $ \mathbf{u}^{\mathrm{opt}}_{(i)} $ and $ \nu^{\mathrm{opt}}_{(i)} $
			
			Computing the corresponding measurements $ \mathbf{y}^{\mathrm{opt}}_{(i)} $ based on the process model $ \mathbf{y} = \mathbf{m}(\mathbf{u}, \mathbf{d}) $ and computing $ \mathcal{G}_{u,(i)} $, and $ \mathcal{L}_{uu,(i)} $
			
		}
		
		k=0
		
		\While{$L>\epsilon$}
		{
			Computing weighted observation loss $ L_w $ (Eq.~\eqref{eq:lw}), achievability loss $ L_a $ (Eq.~\eqref{eq:la}) and uniqueness loss $ L_u $ (Eq.~\eqref{eq:lu}) for $i = 1,2,..,N$
			
			Computing total loss $ L =  \dfrac{1}{N}\sum_{i=1}^{N} (L_{w,i} + L_{a,i} + L_{u,i}) $
			
			$\theta_{k+1}\leftarrow$ Update the parameters of the CVs model $ \theta $ according to the loss $L$ using an optimizer (e.g., Adam)
		}
		$\theta_{\text{opt}}\leftarrow\theta_{k}$ 
	\end{algorithm}

	\section{Case study}
	\label{sec:case}
	{The two case studies presented in this paper serve distinct purposes and demonstrate different aspects of our proposed method:
	\begin{itemize}
		\item \textbf{Case Study 1: Williams-Otto reactor} This is a small-scale case designed to provide a comprehensive comparison between different methods. Its manageable size allows us to implement and compare the optimization-based method, the regression-based method, and our proposed OGCVL method. 
		\item \textbf{Case Study 2: Gold Cyanidation Leaching Process} This is a large-scale, industrially relevant case study.
	\end{itemize}}
	{It should be noted that both examples were chosen primarily to demonstrate the theoretical capabilities of the proposed method. They serve as a proof-of-concept that can be extended to more practical scenarios.}
	
	\subsection{Case 1: Williams-Otto reactor}
	This model is first described in \cite{williams1960reactor}.
	This model consists of one reactor and one heat exchanger. 
	The Williams-Otto reactor converts the raw materials A and B to products P and E, along with by-products C and G, through a series of reactions
	$$\begin{matrix}
		A+B&\rightarrow &C&k_1=k_{10}e^{-6666.7/T_r},k_{10}=1.6599\times{10}^6\\
		B+C&\rightarrow &P+E&k_2=k_{20}e^{-8333.3/T_r},k_{20}=7.2177\times{10}^8\\
		C+P&\rightarrow &G&k_3=k_{30}e^{-11111/T_r},k_{30}=2.6745\times{10}^{12}\\
	\end{matrix}$$	
	The process model could be formulated as,
	\begin{align}
		W\frac{\d x_A}{\d t} &= F_A -  F x_A - W x_A x_B k_1\\
		W\frac{\d x_B}{\d t} &= F_B -  F x_B - W x_A x_B k_1 - W x_B x_C k_2\\
		W\frac{\d x_C}{\d t} &= - F x_C + 2 W x_A x_B k_1 - 2 W x_B x_C k_2 - W x_C x_P k_3\\
		W\frac{\d x_P}{\d t} &= - F x_P + W x_B x_C k_2 - 0.5 W x_P x_C k_3\\
		W\frac{\d x_E}{\d t} &= - F x_E + 2 W x_B x_C k_2\\
		W\frac{\d x_G}{\d t} &= - F x_G + 1.5 W x_P x_C k_3
	\end{align}
	where the mass $ W = 2105~\mathrm{kg} $, the flowrate of feed stream with pure B component $ F_B $ and the reactor temperature $ T_r $ are the manipulated variables. 
	And the level are assumed controlled perfectly such that $ F = F_A+F_B $, where $F_A$ is the feed flowrate of component A and $F$ is the outflow rate.
	
	The objective of this process is to maximize the profits from the valuable products P and E, subject to the constrains on $x_{G_{\max }}$, $x_{A_{\max }}$,  the maximum purity of G and A in the product stream and $F_{min}$, the minimum outflow rate.
	So the steady-state optimization problem is formulated as,
	\begin{equation}
		\begin{array}{cc}
			\min _{T_r, F_B} & -1043.38 x_P F-20.92 x_E F \\
			& \quad+79.23 F_A+118.34 F_B \\
			\text { s.t. } & \\
			& x_G \leqslant x_{G_{\max }} \\
			& x_A \leqslant x_{A_{\max }} \\
			& F \geqslant F_{\min }
		\end{array}
	\end{equation}
	where $ x_{G_{\max }} = 0.08 $, $ x_{A_{\max }}=0.12 $ and $ F_{\min }=4.4~ \mathrm{kg/s}$.
	 
	 In this case, $ F_A $, the flowrate of feed stream with pure A component is the disturbance to the process which is expected to vary between $ 1~\mathrm{kg/s} $ and $ 2~\mathrm{kg/s} $.
	
	\begin{figure}[h]
		\begin{center}
			\includegraphics{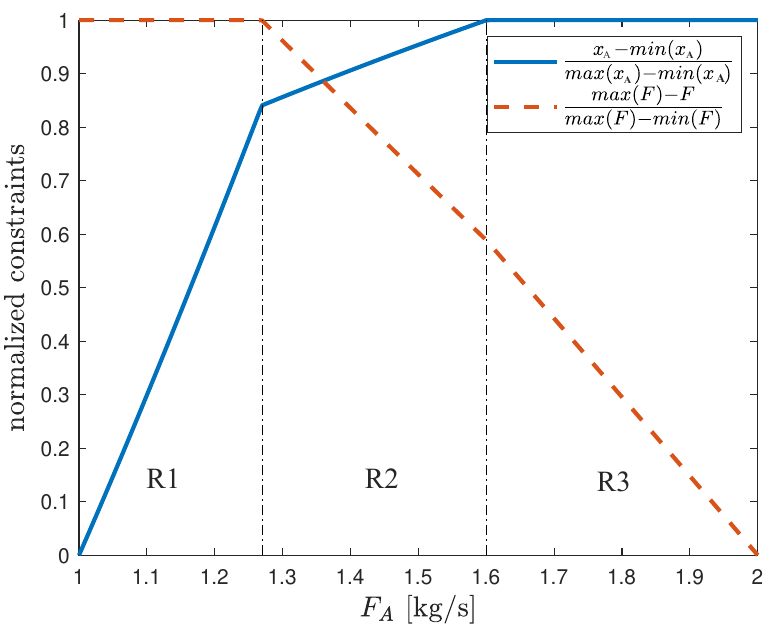} 
			\caption{Variation of active constraints} 
			\label{fig:varyC}
		\end{center}
	\end{figure}
	
	Under the assumed disturbance range of $ F_A \in [1,2] \mathrm{kg/s}$, the purity constraint on $ x_G $ is always active.
	As shown in Fig~\ref{fig:varyC}, the relevant active constraint combinations are 
	\begin{itemize}
		\item $ x_G $ and $ F $ active (R1)
		\item only $ x_G $ active (R2)
		\item $ x_G $ and $ x_A $ active (R3)
	\end{itemize}
	
	Since there are two manipulated variables $ \mathbf{u} = [F_B,T_r] $, firstly, the reactor temperature $ T_r $ is employed to hold $ x_G $ to its upper limit $ x_{G\max} =0.08$ using PI controller.
	The remaining one degree of freedom $ F_B $ will be used to control SOC CV.	
	
	\subsubsection{CV identifications and results}
	
	Firstly, 200 data points are generated using Monte Carlo sampling, for disturbance $ F_A \in [1,2] $. Half of them are used as the training set and the other half as the test set.
	Then, following Algorithm~\ref{algo:OptSOC}	and with the choice of a 1-layer neural network with 10 neurons whose activation function is ReLU, we obtain nonlinear CVs.
	
	\begin{figure}[h]
		\begin{center}
			\includegraphics{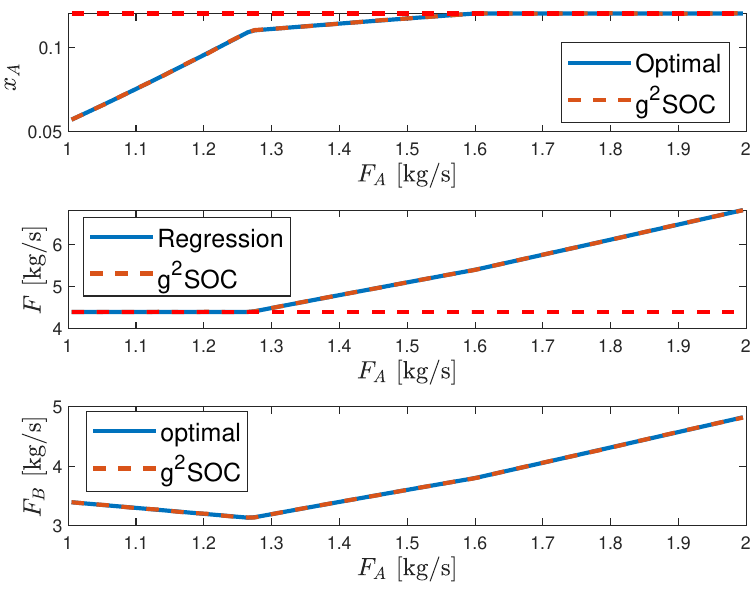} 
			\caption{Close loop performance} 
			\label{fig:cp}
		\end{center}
	\end{figure}
	As shown in Fig~\ref{fig:cp}, the performance of the $g^2$SOC method closely aligns with the optimal solution, essentially achieving near-optimal feedback control. This suggests that the ideal self-optimizing controlled variables have been identified.
	The top and middle subfigs depict the constraint violation, with both the $g^2$SOC method and the optimal solution exhibiting minimal deviations from the constraint boundaries across the varying feed rate ($F_A$). This adherence to constraints is crucial for ensuring safe and reliable operation.
	The bottom subfig shows the control input, further reinforcing the similarity between the $g^2$SOC method and the optimal solution, with their respective control actions exhibiting near-identical profiles over the considered range of feed rates.
	Overall, the graphical comparison unequivocally demonstrates the remarkable performance of the proposed $g^2$SOC method, which seamlessly emulates the ideal optimal feedback control strategy, effectively realizing the self-optimizing control envisioned for this process.

	\begin{table}[]
		\centering
		\caption{Performance indices}
		\label{tab:indices}
		\begin{tabular}{ll}
			\hline
			Index                  & Description                       \\ \hline
			$T_{\mathrm{total}}$(min) & Total time required to train CVs offline                       \\
			$T_{\mathrm{train}}$(min) & Time required to training CVs using the optimization algorithm \\
			$T_{\mathrm{sample}}$(min)  & Time required to generate samples \\
			$Loss_{\mathrm{mean}}$ & Average Economic Loss             \\
			$Loss_{\mathrm{max}}$  & Maximum Economic Loss            \\
			$PO_{\mathrm{max}}$       & Max percentage overshoot of constraints in the violated cases \\ \hline
		\end{tabular}
	\end{table}
	
	In this section, the performance of the optimization-based approach, regression approach and the proposed approaches are compared. The performances of each approach were analyzed through the performance indices listed in Table \ref{tab:indices}.
	The computational resources utilized for the case study are specified, comprising a 64-bit Windows 10 computer with 8 GB of RAM and a 2.9 GHz Intel Core i7-10700F processor. The software packages employed include MATLAB R2023a and the CasADi algorithmic differentiation version 3.5.5\cite{andersson2019}, which incorporates the IPOPT package\cite{biegler2009largeIPOPT}.
	 For the optimization-based approach, the IPOPT solver from the CasADi package was employed. The objective was to minimize the average cost over the one hundred disturbance scenarios used for training, by solving the optimization problem represented by Eq. \eqref{eq:soc}. However, due to the non-convex nature of the problem, it is challenging to obtain a satisfactory solution in practice. In the case of this study, a sufficiently good solution was not successfully obtained. Subsequently, the performances of the regression method and the three components proposed in this work are comparatively evaluated and the results are presented in Table \ref{tab:case1}. 
	 When evaluating the performance of the controlled variables obtained by optimizing the objective function $L_w+\lambda L_a$, the hyperparameter $\lambda = 0.001$ was employed. For assessing the performance of the controlled variables resulting from the objective function $L_w+\lambda L_a + \mu L_u$, the hyperparameters $\lambda= 0.0005, \mu=0.0005$ were utilized.
	 {\begin{remark}[Guidelines for hyperparameter selection ]
	 	The regular term is a soft constraint and is not required to be satisfied, empirically, small numbers (positive numbers less than 1) are usually chosen, and the hyperparameter is usually chosen with reference to the orders of magnitude of $L_w$, $L_a$, and $L_u$, so that all three end up with similar orders of magnitude, so as not to optimize only one of them.
	 \end{remark}}
	 
	\begin{table}[]
		\centering
		\caption{Comparison of performance indices in case 1}
		\label{tab:case1}
		\begin{tabular}{lrrrrrr}
			\hline
			Method &
			$T_{\mathrm{total}}$(min) &
			$T_{\mathrm{train}}$(min) &
			$T_{\mathrm{sample}}$(min) &
			$Loss_{\mathrm{mean}}$ &
			$Loss_{\mathrm{max}}$ &
			$PO_{\mathrm{max}}$ \\ \hline
			Optimization method &
			5756.74 &
			5756.74 &
			0 &
			\multicolumn{1}{c}{-} &
			\multicolumn{1}{c}{-} &
			\multicolumn{1}{c}{-} \\
			$L_o$                       & 1.62 & 1.57 & 0.05 & 0.0437 & 0.3022 & 0.87\% \\
			$L_w$                       & 1.81 & 1.67 & 0.15 & 0.0079 & 0.0679 & 0.11\% \\
			$L_w+\lambda L_a$           & 1.89 & 1.75 & 0.15 & 0.0057 & 0.0881 & 0.17\% \\
			$L_w+\lambda L_a + \mu L_u$ & 2.74 & 2.59 & 0.15 & 0.0042 & 0.0333 & 0.19\% \\ \hline
		\end{tabular}
	\end{table}
	From the overall time spent on offline training CV, the time required for the optimization method is the longest, taking thousands of times as long as the regression method and the method proposed in this paper for the same number of iterations. Simultaneously, it is worth mentioning that in case 1, the problem scale is relatively small, and the disturbances and control inputs are one-dimensional. Even in such a small-scale instance, the time required is excessively prolonged, and the numerical calculations are exceedingly intricate, making the expansion to large-scale problems virtually impossible. 
	
	Considering $Loss_{\mathrm{mean}}$, the effect of weighting is most pronounced, with the weighted average economic loss significantly diminished. Both the achievability and uniqueness regularizers have also played a role, reducing the average loss. Additionally, it can be observed that the uniqueness regularizer has a comparatively noticeable effect on minimizing $Loss_{\mathrm{max}}$. Regardless of the methods employed, none have achieved perfect constraint satisfaction, but the quantity of constraint violations $PO_{\mathrm{max}}$ is minuscule, which can be addressed by setting a certain amount of backoff or employing a switching control structure method to satisfy the constraints \cite{ye2023}.

	{This case study illustrates:
	\begin{itemize}
		\item The computational challenges faced by the optimization-based method, particularly the long solution times even for small-scale problems.
		\item The limitations of the pure regression-based approach in terms of control performance.
		\item The balanced performance of our proposed OGCVL method in terms of both computational efficiency and control performance.
	\end{itemize}}	
	
	\subsection{Case 2: Gold Cyanidation Leaching Process}
	This example pertains to the gold cyanidation leaching process at a certain hydrometallurgical plant \cite{yeRealTimeOptimizationGold2017}. The leaching process is a continuous operation, involving four overflow-connected air agitation leaching tanks. Sodium cyanide serves as the leaching agent, and compressed air is continuously blown into the leaching tanks by fans, providing the dissolved oxygen required for the leaching process while also generating air agitation to promote more thorough leaching reactions. The primary reaction is:
	
	$$ 4 \mathrm{Au}+8 \mathrm{CN}^{-}+\mathrm{O}_2+2 \mathrm{H}_2 \mathrm{O} \rightarrow 4 \mathrm{Au}(\mathrm{CN})_2^{-}+4 \mathrm{OH}^{-} $$
	
	The mechanistic model of the leaching process can be established based on the law of mass conservation, including the conservation of gold elements in the solid and liquid phases, the conservation of cyanide ions, and the reaction rate equations, namely:
	
	\begin{equation} 
		\label{eq:dyCase2}
		\begin{aligned} 
			& \frac{\mathrm{d} C_{\mathrm{s}, i}}{\mathrm{~d} t}=\frac{Q_{\mathrm{s}, i}}{M_{\mathrm{s}, i}}\left(C_{\mathrm{s}, i-1}-C_{\mathrm{s}, i}\right)-r_{\mathrm{Au}, i} \\ 
			& \frac{\mathrm{d} C_{\mathrm{l}, i}}{\mathrm{~d} t}=\frac{Q_{\mathrm{l}, i}}{M_{\mathrm{l}, i}}\left(C_{\mathrm{l}, i-1}-C_{\mathrm{l}, i}\right)+\frac{M_{\mathrm{s}, i}}{M_{\mathrm{l}, i}} r_{\mathrm{Au}, i} \\ 
			& \frac{d C_{\mathrm{CN}, i}}{d t}=\frac{Q_{\mathrm{l}, i}}{M_{\mathrm{l}, i}}\left(C_{\mathrm{CN}, i-1}-C_{\mathrm{CN}, i}\right)+\frac{Q_{\mathrm{CN}, i}}{M_{\mathrm{l}, i}}-r_{\mathrm{CN}, i} \\
			& r_{\mathrm{Au}, i}=k_1\left(C_{\mathrm{s}, i}-C_{\mathrm{s \infty}}\right)^{k_2} C_{\mathrm{CN}, i}^{k_3} C_{\mathrm{o}, i}^{k_4} \\ 
			& r_{\mathrm{CN}, i}=k_5 C_{\mathrm{CN}, i}^{k_6} \\
			& C_{\mathrm{s} \infty}=0.357\left(1-1.49 \mathrm{e}^{-0.0176 \bar{d}}\right)\\
			& Q_{\mathrm{s}, i}=Q_{\mathrm{s}, 0} \\
			& Q_{\mathrm{l}, i}=Q_{\mathrm{l},0} \\
			& Q_{\mathrm{l}, i}=\frac{1-C_{\mathrm{w}}}{C_{\mathrm{w}}} Q_{\mathrm{s}, i} \\
			& V=\frac{M_{\mathrm{s}, i}}{\rho_{\mathrm{s}}}+\frac{M_{1, i}}{\rho_1} \\
			& \frac{M_{\mathrm{s}, i}}{M_{1, i}}=\frac{Q_{\mathrm{s}, i}}{Q_{\mathrm{l}, i}} \\
			& \tau_i=\frac{V}{\frac{Q_{s, i}}{\rho_{\mathrm{s}}}+\frac{Q_{\mathrm{l},i}}{\rho_1}}=\frac{M_{1, i}}{Q_{\mathrm{l}, i}}=\frac{M_{\mathrm{s}, i}}{Q_{\mathrm{s}, i}}
		\end{aligned}
	\end{equation}
	where $ i=1,\dots,4$ denotes $i$-th stage, and $i =0$ denotes an initial value.			
	The meanings and nominal values of each symbol in the process are introduced in Table \ref{table:case2P}.
	\begin{table}
		
		\caption{Parameter Values of GCLP}
		\centering
		\begin{tabular}{lll}
			\hline 
			parameter variable & description & value \\
			\hline
			$C_{\mathrm{CN}, 0}$ & initial cyanide concentration in liquid & $200  \mathrm{mg} / \mathrm{kg}$  \\
			$C_{\mathrm{l}, 0}$ & initial gold concentration in liquid & $0 \mathrm{mg} / \mathrm{kg}$ \\
			$C_{\mathrm{s}, 0}$ & initial gold concentration in ore & $100 \mathrm{mg} / \mathrm{kg}$ \\
			$C_{\mathrm{o}}$ & oxygen concentration in solution & $7 \mathrm{mg} / \mathrm{kg}$ \\
			$C_{\mathrm{w}}$ & solid concentration in the pulp & $0.39 \mathrm{~kg} / \mathrm{kg}$ \\
			$P_{\mathrm{Au}}$ & gold price & $0.226 \mathrm{CNY} / \mathrm{mg}$ \\
			$P_{\mathrm{CN}}$ & cyanide price & $15 \times 10^{-6} \mathrm{CNY} / \mathrm{mg}$ \\
			$P_{\mathrm{CNd}}$ & price for cyanide destruction & $2.5 \times 10^{-6} \mathrm{CNY} / \mathrm{mg}$ \\
			$Q_{\mathrm{s}}$ & ore flow rate & $2500 \mathrm{~kg} / \mathrm{h}$ \\
			$V$ & reactor holdup & $68  \mathrm{~m}^3$ \\
			$\bar{d}$ & average size of the ore particles & $139 \mu \mathrm{m}$ \\
			$k_1$ & kinetic parameter & 0.0011 \\
			$k_2$ & kinetic parameter & 2.13 \\
			$k_3$ & kinetic parameter & 0.961 \\
			$k_4$ & kinetic parameter & 0.228 \\
			$k_5$ & kinetic parameter & $3.6821 \times 10^{-9}$ \\
			$k_6$ & kinetic parameter & 3.71 \\
			$\rho_{\mathrm{s}}$ & ore density & $2.8 \mathrm{~g} / \mathrm{cm}^3$ \\
			$\rho_1$ & liquid density & $1 \mathrm{~g} / \mathrm{cm}^3$ \\
			$a^*$ & minimal gold recovery rate & $90 \%$ \\
			$Q_{\mathrm{CN}, \text { max }}$ & maximal cyanide supply & $10 \mathrm{~kg} / \mathrm{h}$ \\
			\hline
		\end{tabular}
		\label{table:case2P}
	\end{table}
	
	The objective of this process is to minimize a cost function $J$ while ensuring that the gold recovery rate $a$ is greater than a minimal level $a^*$ by adjusting flow rates $\mathbf{u}=[Q_{\mathrm{CN}, 1},Q_{\mathrm{CN}, 2},Q_{\mathrm{CN}, 3},Q_{\mathrm{CN}, 4}]$.
	The optimization problem is formulated as follows:
	\begin{equation}
		\begin{array}{ll} 
			& \min_{\mathbf{u}} J=P_{\mathrm{CN}}\left(\sum_{i=1}^4 Q_{\mathrm{CN}, i}+C_{\mathrm{CN}, 0} Q_{\mathrm{l},0}\right) \\
			& +P_{\mathrm{CNd}} Q_{\mathrm{l}, n}*C_{\mathrm{CN},4}+P_{\mathrm{Au}} Q_{\mathrm{s}, n} C_{\mathrm{s}, n} \\
			\text { s.t. } & \text {process model: }\eqref{eq:dyCase2} \\
			& 0 \leq Q_{\mathrm{CN}, i} \leq Q_{\mathrm{CN}, \max }, \quad \forall i=1, \ldots, 4 \\
			& a = \dfrac{C_{\mathrm{s},0}-C_{\mathrm{s},4}}{C_{\mathrm{s},0}} \geq a^* 
		\end{array}
	\end{equation}
	The expected disturbances for the GCLP encompass parametric uncertainties, wherein the uncertain kinetic parameters $k_1-k_6$ are randomly distributed within ±20\% of their nominal values, and the gold price $P_{Au}$ varies within the range of 190-280 CNY/g.
	
	This case involves a large-scale system with numerous state variables, exhibiting strong non-linear dynamics, subject to high-dimensional and broad-ranging disturbances, encompassing multiple degrees of freedom. Employing optimization-based techniques would typically be infeasible. Subsequently, the viability of the $g^2$SOC methodology for this large-scale case will be elucidated.
	
	In this example, the considered measurement variables encompass flow rate, the gold concentration in the solid phase, gold concentration in the liquid phase, and cyanide concentration in the liquid phase across four stages, as well as the gold price.
	$$\mathbf{y} = \left[ C_{\mathrm{s}, i},C_{\mathrm{l}, i},C_{\mathrm{CN}, i},Q_{\mathrm{CN}, i},P_{Au}\right],i=1,\dots,4$$
	
	Firstly, 2500 sample scenarios were generated via Monte Carlo sampling under the disturbances $P_{Au}$ and $k_1-k_6$. Subsequently, the optimization problem was solved individually for these 2500 scenarios, yielding the optimal control inputs, and simulations were performed to obtain the corresponding optimal measurement data. Thereby, the dataset was constructed. The initial 2400 scenarios in the dataset were employed for training the self-optimizing controlled variables, while the remaining 100 scenarios were reserved for testing purposes. The controlled variables were trained using Algorithm~\ref{algo:OptSOC}.
	The neural network architecture with three hidden layers, each comprising 15 neurons with ReLU, was employed.
	
	\begin{table}[]
		\centering
		\caption{Comparison of performance indices in case 2}
		\label{tab:case2}
		\begin{tabular}{lrrrrr}
			\hline
			& $T_{\mathrm{total}}$(min) & $T_{\mathrm{train}}$(min) & $T_{\mathrm{sample}}$(min) & $Loss_{\mathrm{mean}}$ & $Loss_{\mathrm{max}}$ \\ \hline
			$L_o$                       & 2.90  & 2.38  & 0.52 & 0.0884 & 1.2519 \\
			$L_w$                       & 6.44  & 4.27  & 2.17 & 0.0405 & 0.5133 \\
			$L_w+\lambda L_a$           & 33.18 & 31.01 & 2.17 & 0.0334 & 0.9442 \\
			$L_w+\lambda L_a + \mu L_u$ & 50.93 & 48.75 & 2.17 & 0.0249 & 0.1272 \\ \hline
		\end{tabular}
	\end{table}
	As shown in Table \ref{tab:case2}, a comparison was drawn with regression-based techniques, as well as the three key components proposed in Objective-guided controlled variables Learning for $g^2$SOC, namely, the weighted, feasibility and uniqueness regularizers. An examination of $Loss_{\mathrm{mean}}$ and $Loss_{\mathrm{max}}$ revealed consistent trends with Case 1, wherein the weighted approach significantly reduced closed-loop economic losses, while the feasibility and uniqueness regularizers effectively mitigated closed-loop losses. Notably, the uniqueness regularizer exhibited a pronounced impact on minimizing maximum economic losses. Concurrently, as the dimensionality of the control inputs increased relative to Case 1, the feasibility and uniqueness regularizers led to a substantial rise in training time. Nevertheless, compared to optimization-based methods, the proposed approach retained its advantageous position.
	
	The aforementioned experiments were conducted from a steady-state perspective, assuming that SOC could effectively control the CVs during a transient phase. To substantiate its feasibility, dynamic simulations are provided. With four degrees of freedom, four PI controllers are employed for individual control in this case. 
	
	In this case, measurement noise are considered as zero-mean Gaussian noise with a standard deviation of 2\textperthousand of the nominal process variable values. The impact of measurement noise on CVs is addressed through filtered measurements. Specifically, an "exponential moving average" filter is utilized, formulated as follows: 
	\begin{equation}
		c_k=\alpha c_{m, k}+(1-\alpha) c_{k-1}, \quad \text{ where }  \alpha=\frac{1}{1+\tau_F / \Delta t_s}
	\end{equation}
	where $\Delta t_s$ is the sampling time, $\tau_F$ is the filter time. In this case, $\Delta t_s = 1$ min, $\tau_F = 999$ min.
	
	The dynamic simulation arrangement is as follows: Initially, the system operates under nominal conditions S1; every 100 h, the GCLP operating conditions are altered to S2 and S3, respectively. The three disturbance scenarios are presented in Table~\ref{tab:disSce}. 
	\begin{table}[]
		\centering
		\caption{Three Disturbance Scenarios}
		\label{tab:disSce}
		\begin{tabular}{lccccccc}
			\hline
			Disturbance   scenario & $k_1$    & $k_2$  & $k_3$   & $k_4$   & $k_5$               & $k_6$  & ${P_{\mathrm{Au}}}$ \\ \hline
			$S_1$(nominal)           & 0.0011  & 2.13 & 0.961 & 0.228 & $3.58\times10^{-9}$ & 3.71 & 0.226             \\
			$S_2$                    & 0.0013  & 2.04 & 0.813 & 0.19  & $3.48\times10^{-9}$ & 4.2  & 0.204             \\
			$S_3$                    & 0.00097 & 1.91 & 0.781 & 0.219 & $3.20\times10^{-9}$ & 3.59 & 0.276             \\ \hline
		\end{tabular}
	\end{table}
	The dynamic response of SOC, as illustrated in Figure \ref{fig:dynamic}, 
	demonstrates that when the controlled variables generated by the OGCVL's $g^2$SOC approach are kept at zeroes, their economic performance closely approximates the optimal scenario, aligning with the conclusions drawn from the previous steady-state analyses.
	
	\begin{figure}
		\centering
		\begin{subfigure}[b]{1\textwidth}
			\includegraphics[width=\textwidth]{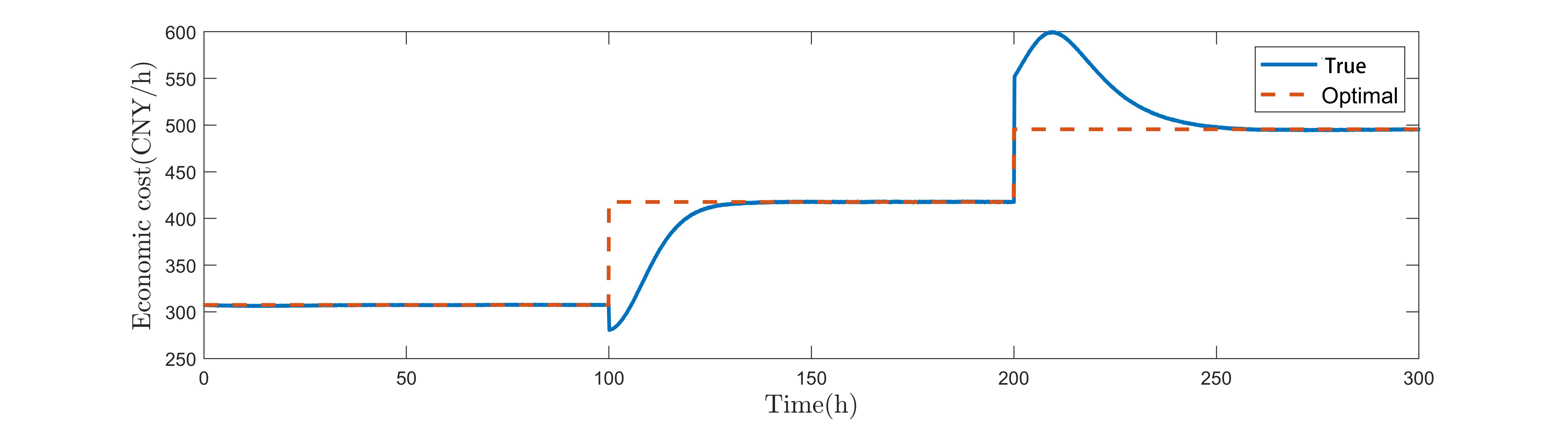}
			\caption{}
		\end{subfigure}
		\begin{subfigure}[b]{0.45\textwidth}
			\includegraphics[width=\textwidth]{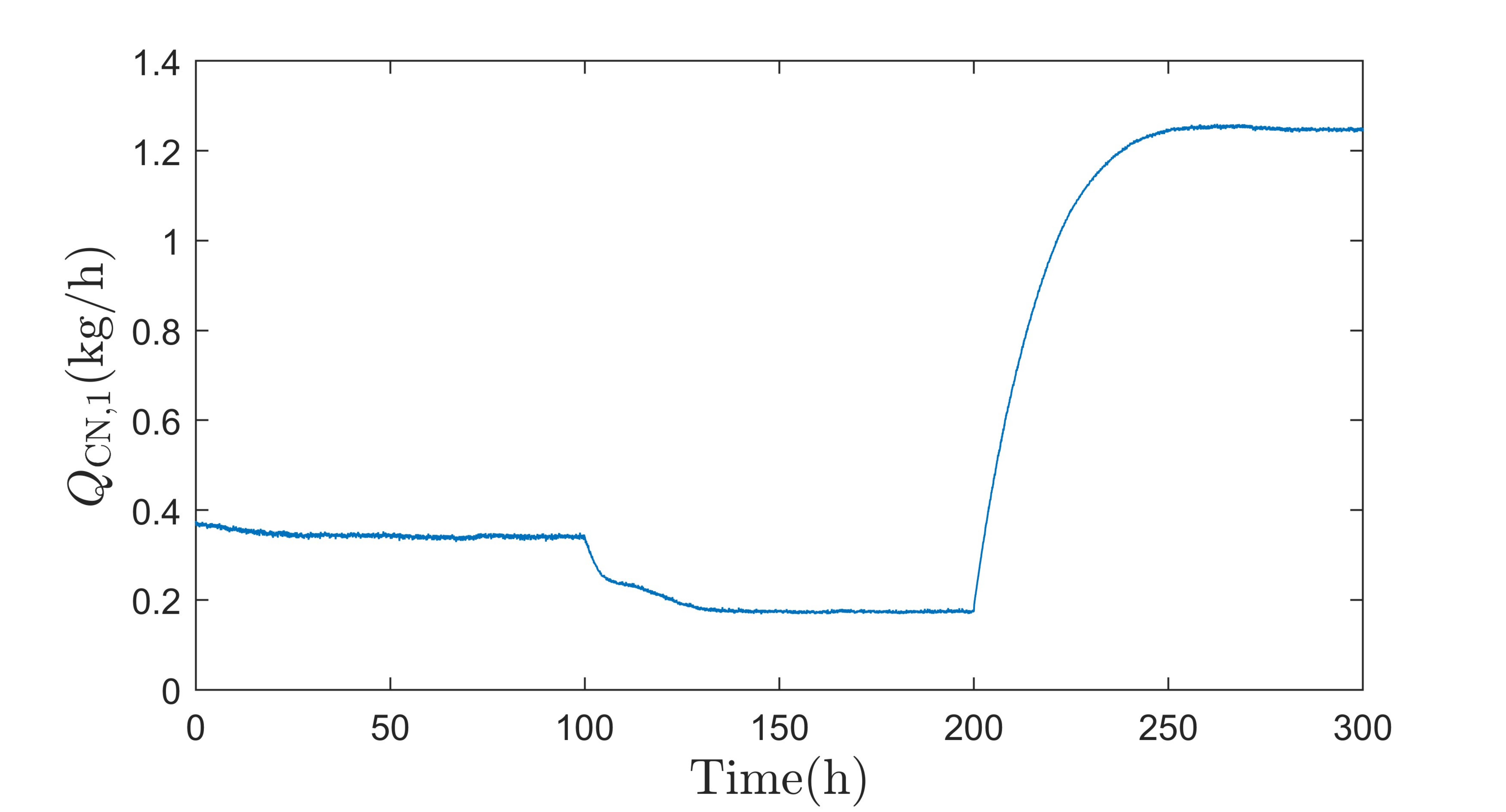}
			\caption{}
		\end{subfigure}
		\begin{subfigure}[b]{0.45\textwidth}
			\includegraphics[width=\textwidth]{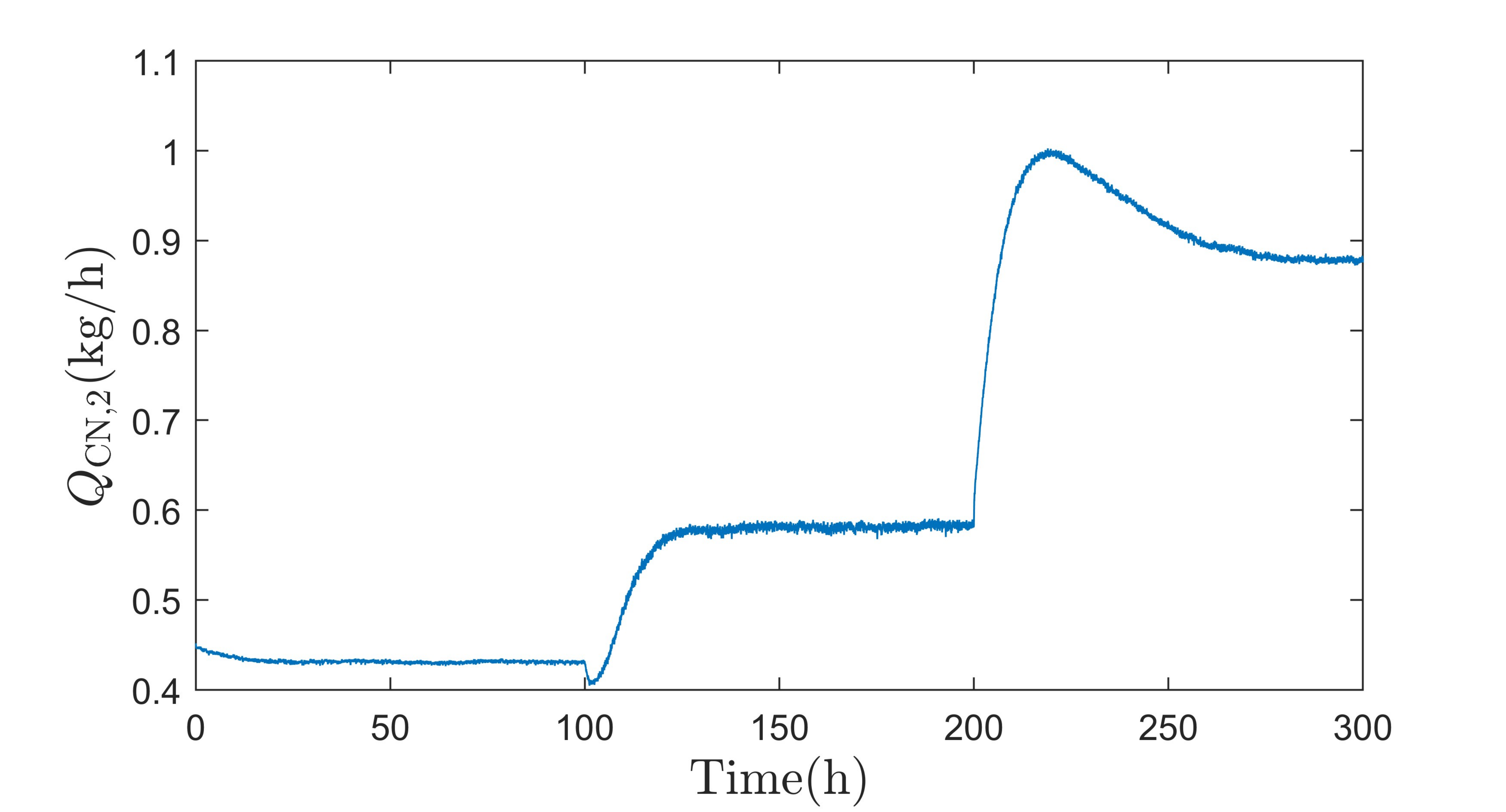}
			\caption{}
		\end{subfigure}
		\begin{subfigure}[b]{0.45\textwidth}
			\includegraphics[width=\textwidth]{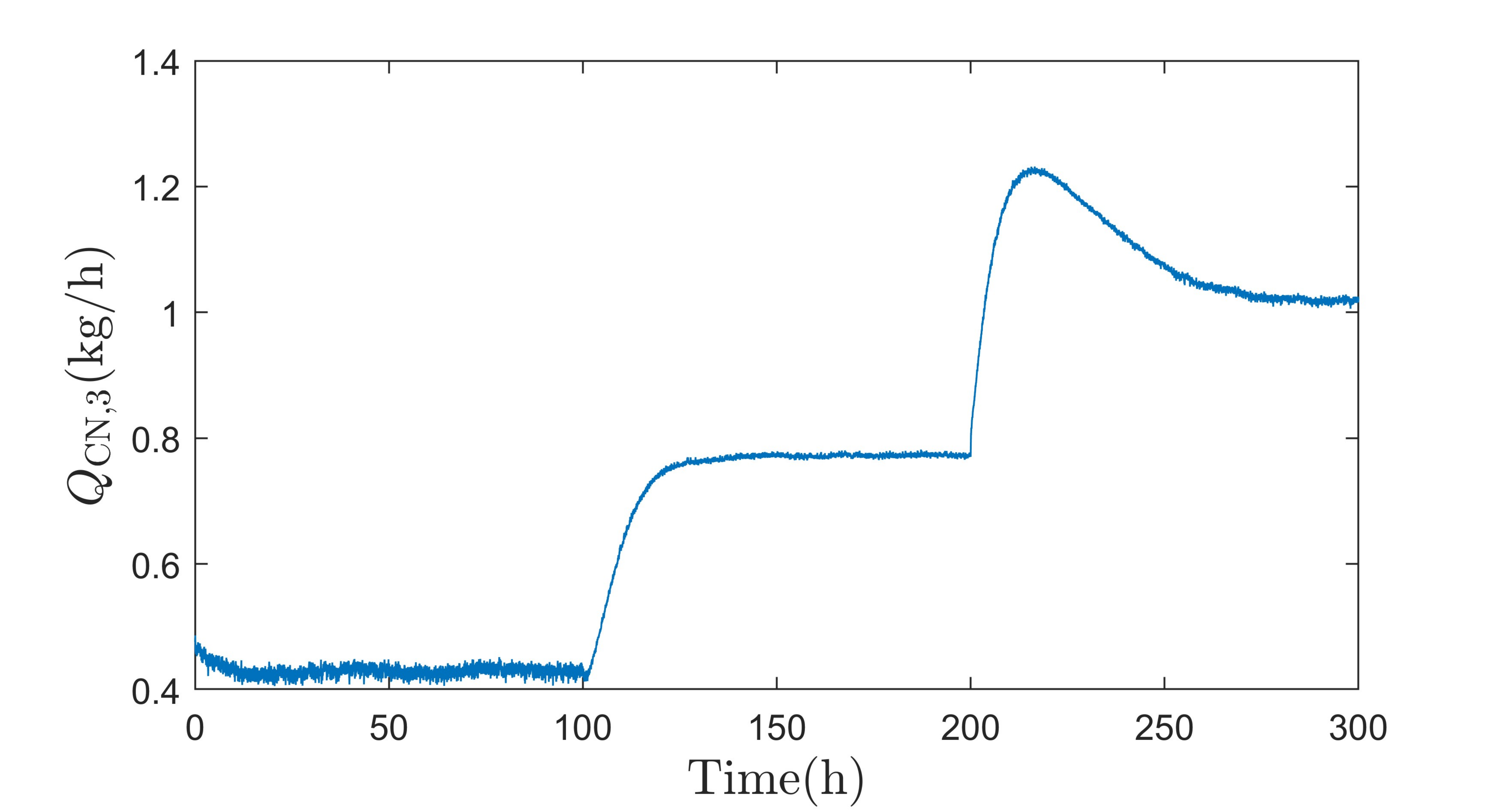}
			\caption{}
		\end{subfigure}
		\begin{subfigure}[b]{0.45\textwidth}
			\includegraphics[width=\textwidth]{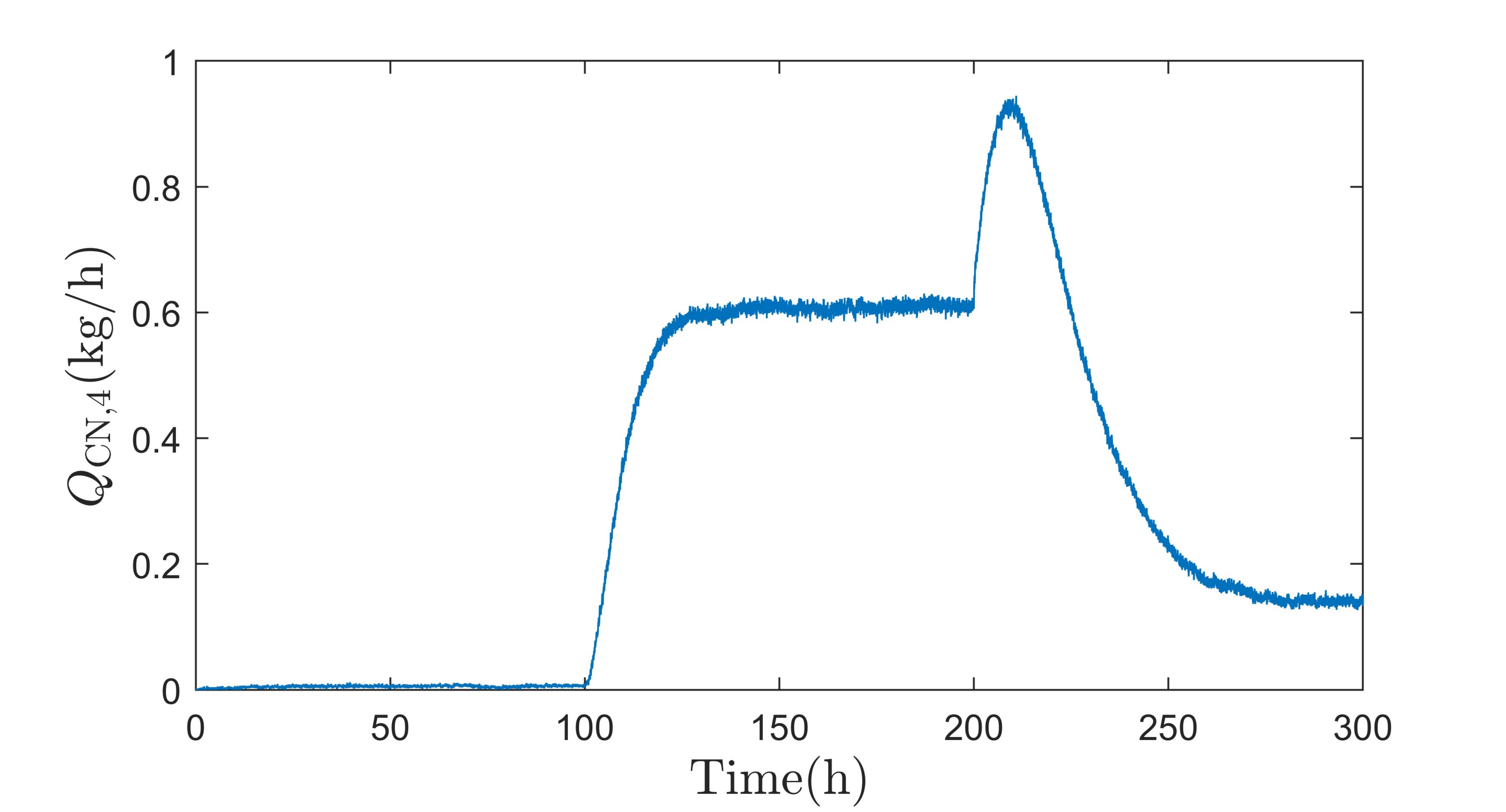}
			\caption{}
		\end{subfigure}
		\caption{Dynamic simulation for the SOC scheme (for every 100 h, the system is operated under S1, S2, and S3, respectively).
		}
		\label{fig:dynamic}
	\end{figure}
	
	{This case study demonstrates:
	\begin{itemize}
		\item The scalability of our proposed OGCVL method to complex, real-world problems.
		\item The method's ability to handle cases where traditional optimization-based approaches become computationally intractable.
		\item The robustness and effectiveness of OGCVL in maintaining near-optimal performance across a wide range of operating conditions in a more challenging, high-dimensional problem space.
	\end{itemize}}
	
	\section{Conclusion and outlooks}	
	\label{sec:con}
	{This work, as Part II of our study on generalized global Self-Optimizing Control ($g^2$SOC), builds upon the theoretical foundations established in Part I \cite{ye2023}. While Part I proved the existence of perfect global self-optimizing controlled variables and proposed two numerical approaches for their approximation, this work addresses the practical challenges in implementing these approaches and presents a novel framework that combines their strengths.
	The Objective-guided Controlled Variable Learning (OGCVL) framework presented here represents a significant advancement in the design of controlled variables (CVs) for self-optimizing control (SOC). By integrating domain knowledge from first-principles process models into the data-driven training process, OGCVL bridges the gap between symbolic and numerical computation techniques. This framework successfully addresses the limitations identified in Part I, particularly the computational inefficiency of the optimization-based approach and the potential suboptimality of the pure regression-based approach.}
	
	{Key improvements over Part I include:
	\begin{itemize}
		\item Enhanced computational efficiency, especially for large-scale problems, while maintaining the ability to capture complex nonlinear relationships.
		\item A systematic approach that balances the trade-off between optimality and computational tractability.
	\end{itemize}}
	
	{The case studies demonstrate the feasibility and effectiveness of the OGCVL algorithm in both small-scale and larger-scale scenarios, validating its potential for practical implementation. 
	{Methods such as model predictive control policy approximation \cite{zhou2024c} can be considered as regression-based methods, and research on this area can draw on the methods proposed in this paper to enhance their performance.}	
	However, it is important to acknowledge the limitations and areas for further research:
	\begin{itemize}
		\item Extension to dynamic optimization problems: The current work, like Part I, focuses on steady-state optimization. Future research should investigate the application of $g^2$SOC and OGCVL to dynamic systems.
		\item Comparative analysis of regularization methods: Further study is needed to compare the gradient-based regularization used in OGCVL with alternative approaches, such as the comparative learning method mentioned in \cite{zhou2024b}.
		\item {Typical measurement limitations: Online measurement of some key physical properties can be a challenge, such as concentration. And when such measurements are possible, they often suffer from high uncertainty and potential gross errors. The method proposed in this paper does not consider the effect of measurement noise when designing the controlled variable, and in future research, strategies such as using data with noise to train the controlled variable can be considered.}
		\item {Exploration of more complex industrial cases: While we have shown the usability of the proposed method in larger scale problems, testing on a wider range of complex, real-world processes would further validate the method's robustness. At the same time, for system-scale problems, it may be necessary to consider a combination of techniques such as parallelization strategies or problem preprocessing techniques.}
		\item {Constraint Satisfaction Under Model Mismatch : An important area for future investigation is the development of methodologies to ensure constraint satisfaction when faced with model-plant mismatch. There have been a number of studies focusing on such issues, such as \cite{krishnamoorthy2020}}
		\item {Addressing practical implementation challenges, particularly the development of hybrid approaches that can leverage both model-based and data-driven methods when high-fidelity models are not readily available. This extension would enhance the framework's applicability across a broader range of industrial scenarios where model accessibility may be limited.}
	\end{itemize}}
	
	This work not only advances the theoretical foundations of self-optimizing control but also offers practical implications for industry, potentially revolutionizing current practices in real-time optimization while providing a framework adaptable to various process industries beyond chemical engineering, particularly in scenarios requiring economic optimization under uncertainty.
	
	{In conclusion, the OGCVL framework represents a significant step forward from Part I in the practical realization of $g^2$SOC. It offers a promising approach for designing high-performance nonlinear CVs, addressing key challenges identified in our previous work. However, continued research efforts are essential to further enhance its capabilities and broaden its applicability across diverse process systems.}

	\bibliography{biblio}
	
	\newpage
	\begin{figure}[h]
		\centering
		\includegraphics[width=0.9\textwidth]{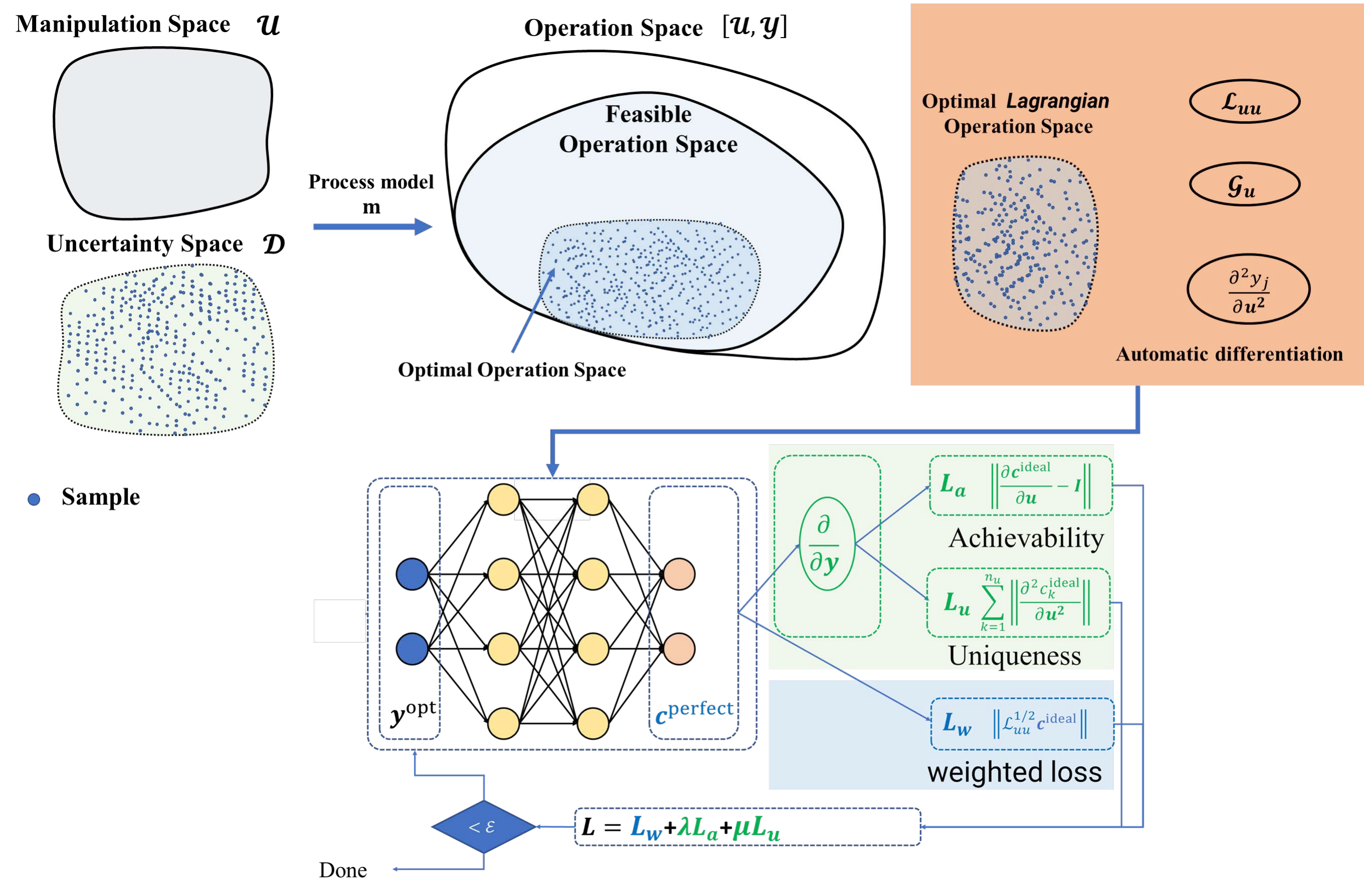} 
	\end{figure}
	
\end{document}